\documentclass[11pt]{amsart}
\usepackage{amscd,amssymb}
\usepackage{amsthm,amsmath,amssymb}
\usepackage[matrix,arrow]{xy}
\usepackage{enumerate}
\newtheorem{theorem}{Theorem}[section]
\newtheorem{proposition}[theorem]{Proposition}
\newtheorem{lemma}[theorem]{Lemma}
\newtheorem{corollary}[theorem]{Corollary}

\newtheorem*{conjecture*}{Conjecture}

\theoremstyle{definition}
\newtheorem{example}[theorem]{Example}
\newtheorem{definition}[theorem]{Definition}

\theoremstyle{remark}
\newtheorem{remark}[theorem]{Remark}

\makeatletter\@addtoreset{equation}{section} \makeatother

\title{Stable polarized del Pezzo surfaces}

\author{Ivan Cheltsov and Jesus Martinez-Garcia}
\address{i.cheltsov@ed.ac.uk\newline
	MSU, Faculty of Mechanics and Mathematics, Russia, 119991, Moscow, GSP-1, 1 Leninskiye Gory, Main Building.\newline
	University of Edinburgh, Department of Mathematics, Mayfield Rd., Edinburgh EH9 3JZ, UK.
	\newline
	J.Martinez.Garcia@bath.ac.uk\newline
	Department of Mathematical Sciences, University of Bath, Claverton Down, Bath BA2 7AY, UK.}

\begin{document}

\begin{abstract}
We give a simple sufficient condition for $K$-stability of polarized del Pezzo surfaces and for the existence of a constant scalar curvature K\"ahler metric in the K\"ahler class corresponding to the polarization.
\end{abstract}

\sloppy

\maketitle

\section{Introduction}
\label{section:into}

In recent years the notion of $K$-stability has been of great importance
in the study of the existence of canonical metrics on complex varieties.
This is mainly because of the following

\begin{conjecture*}[Yau--Tian--Donaldson]
\label{conjecture:YTD}
Let $X$ be a smooth variety, and let $L$ be an ample line bundle on $X$.
Then $X$ admits a constant scalar curvature K\"ahler (cscK) metric in $c_1(L)$
if and only if the pair $(X,L)$ is $K$-polystable.
\end{conjecture*}

It is known in different degrees of generality that $K$-stability is a necessary condition for the existence of a cscK metric,
with the most general result due to Berman, Darvas and Lu \cite{BDL16} following work of Darvas and Rubinstein \cite{DR17}.
To show that it is also a sufficient condition is currently one of the main open questions in the field.
For smooth Fano varieties polarized by anticanonical line bundles,
this was recently proved by Chen, Donaldson and Sun in \cite{CDS}.

The goal of this paper is to study $K$-stability of polarized smooth del Pezzo surfaces.
This problem is explicitly solved in the toric case by Donaldson \cite{SD-toric-surfaces}.
Surprisingly, we do not know many results about the non-toric case.
By the famous theorem of Tian \cite{Tian1990}, all non-toric smooth del Pezzo surfaces are K\"ahler-Einstein,
so that they are $K$-stable for the anticanonical polarization.
Results of Arezzo and Pacard \cite{AP1,AP2}, Arezzo, Pacard and Singer \cite{APS}, and Rollin and Singer \cite{RS} imply the $K$-stability for many other polarizations.
On the other hand, one can use Ross and Thomas's \cite[Example~5.30]{RT} to produce many $K$-unstable polarizations
on every smooth non-toric del Pezzo surface.

The main result of this paper is

\begin{theorem}
\label{theorem:main}
Let $S$ be a smooth del Pezzo surface such that $K_S^2\leqslant 2$, and let $L$ be an ample $\mathbb{Q}$-divisor on $S$.
If $-K_S-\frac{2}{3}\frac{-K_S\cdot L}{L^2}L$ is nef, then $(S,L)$ is $K$-stable.
\end{theorem}

For each smooth del Pezzo surface of degree one or two, this result provides a closed subset in its ample cone that consists of $K$-stable polarizations.
This subset contains an anticanonical divisor, so that Theorem~\ref{theorem:main} also implies
the existence of K\"ahler-Einstein metrics on smooth del Pezzo surfaces of degree one and two.
The nefness condition in Theorem~\ref{theorem:main} is easy to check,
because the Mori cones of these del Pezzo surfaces are generated by finitely many $(-1)$-curves. The polarizations satisfying the nefness condition are far from those considered by Arezzo, Pacard, Rollin and Singer \cite{AP1,AP2,APS,RS}, and therefore they are new. 

Combining Theorem \ref{theorem:main} with \cite{Chen-Cheng} by Chen and Cheng, we obtain the following
\begin{corollary}
Let $S$ be a smooth del Pezzo surface such that $K_S^2\leqslant 2$, and let $L$ be an ample $\mathbb{Q}$-divisor on $S$.
If $-K_S-\frac{2}{3}\frac{-K_S\cdot L}{L^2}L$ is nef, then $S$ admits a constant scalar curvature K\"ahler metric in the K\"ahler class $c_1(L)$.
\end{corollary}

For smooth del Pezzo surfaces of degree one, the assertion of Theorem~\ref{theorem:main} was recently proved by Hong and Won in \cite{HongWon} using a different approach.

The $\alpha$-invariant of Tian, originally defined in terms of complex differential geometry,
has an algebraic formulation which we use to study the $K$-stability of the anticanonically polarized del Pezzo surfaces.
For a variety $X$ and an ample $\mathbb{Q}$-divisor $L$ on it, we let
$$
\alpha\big(X,L\big)=\mathrm{sup}\left\{\lambda\in\mathbb{Q}\ \left|%
\aligned
&\text{the log pair}\ \left(X, \lambda D\right)\ \text{is log canonical}\\
&\text{for every effective $\mathbb{Q}$-divisor}\ D\sim_{\mathbb{Q}} L
\endaligned\right.\right\}\in\mathbb{R}_{>0}.%
$$
While the $\alpha$-invariant is very difficult to estimate for an arbitrary polarization,
it is the only effective invariant to provide sufficient condition for $K$-stability.
In this article we will prove Theorem~\ref{theorem:main} using the following criterion of Dervan:

\begin{theorem}[{\cite[Theorem~1.1]{Dervan1}}]
\label{theorem:Dervan}
Let $X$ be a Fano variety with log canonical singularities, and let $L$ be an ample $\mathbb Q$-divisor on it.
Let $\nu(L)=\frac{-K_X\cdot L^{n-1}}{L^n}$.
Then $(X,L)$ is $K$-stable provided that
\begin{itemize}
\item[$(\mathrm{A})$] the $\mathbb Q$-divisor $-K_X-\frac{n}{n+1}\nu(L)L$ is nef,
\item[$(\mathrm{B})$] and $\alpha(X,L)>\frac{n}{n+1}\nu(L)$.
\end{itemize}
\end{theorem}

Namely, we will show that for smooth del Pezzo surfaces of degree one and two,
the condition $(\mathrm{B})$ in this theorem follows from the condition $(\mathrm{A})$.
The reader may wonder whether the condition $(\mathrm{B})$ in Theorem~\ref{theorem:Dervan} is redundant.
In general, this is not the case:

\begin{example}
\label{example:cubic-surface}
Let $S$ be a smooth cubic surface in $\mathbb{P}^3$ that does not contain Eckardt points, and let $E$ be a line in $S$.
Let $L=-K_{S}+xE$, where $x$ is a non-negative rational number such that $x<1$.
Then $L$ is ample and
$$
\frac{-K_S\cdot L}{L^2}=\frac{3+x}{3+2x-x^2},
$$
so that $-K_{S}-\frac{2}{3}\frac{-K_S\cdot L}{L^2}L$ is nef if and only if $x\leqslant\frac{3}{5}$.
On the other hand, we have $(1+x)E+C\sim_{\mathbb Q} L$, where $C$ is a conic in $|-K_S-E|$ that is tangent to $E$.
This immediately gives $\alpha(S,L)\leqslant\frac{3}{4+2x}$.
Thus, if $\frac{3}{5}\geqslant x\geqslant\frac{-1+2\sqrt{10}}{13}$, then $-K_{S}-\frac{2}{3}\frac{-K_S\cdot L}{L^2}L$ is nef,
while $\alpha(S,L)\leqslant\frac{2}{3}\frac{-K_S\cdot L}{L^2}$.
\end{example}

In this paper, we also prove the following negative result.

\begin{theorem}
\label{theorem:big-degree}
Let $S$ be a smooth del Pezzo surface such that $K_S^2\geqslant 4$, and let $L$ be an ample $\mathbb{Q}$-divisor on $S$.
Then
$\alpha(S,L)\leqslant\frac{2}{3}\frac{-K_S\cdot L}{L^2}$.
Moreover, the equality holds if and only if $K_S^2=4$ and $L\in\mathbb{Q}_{>0}[-K_{S}]$.
\end{theorem}

Hence Theorem~\ref{theorem:main} exhausts the application of Theorem \ref{theorem:Dervan} to all smooth del Pezzo surfaces other than cubic surfaces.
Fortunately, we can still obtain partial results in the latter case.

\begin{theorem}
\label{theorem:cubic-surface-1}
Let $S$ be a smooth cubic surface in $\mathbb{P}^3$, and let $E_1$, $E_2$, $E_3$, $E_4$, $E_5$, $E_6$ be disjoint lines on the surface $S$.
Let $L=-K_{S}+x\sum_{i=1}^6E_i$, where $x$ is a non-negative rational number such that~$x<1$.
Then $L$ is ample. Furthermore, if $0<x\leqslant\frac{1}{10}$, then $(S,L)$ is $K$-stable.
\end{theorem}

Since all smooth cubic surfaces are K\"ahler-Einstein by Tian's theorem \cite{Tian1990},
the $K$-stability of the pair $(S,L)$ in this theorem holds also for $x=0$.
Moreover, Arezzo and Pacard's results \cite{AP2} imply the $K$-stability of the pair $(S,L)$ in the case when $x$ is sufficiently close to $1$.
Thus, one may expect that the pair $(S,L)$ is $K$-stable for all positive rational numbers $x<1$.
Unfortunately, we were unable to prove this.

Let us describe the structure of this paper.
In Section~\ref{section:background} we present some auxiliary results.
In Section~\ref{section:degree-one}, we give a very short proof of our Theorem~\ref{theorem:main} for smooth del Pezzo surfaces of degree one.
In Section~\ref{section:degree-two}, we prove Theorem~\ref{theorem:main} for smooth del Pezzo surfaces of degree two.
This section is the main part of the paper.
Then we prove Theorem~\ref{theorem:cubic-surface-1} in Section~\ref{section:cubic-surfaces}.
Finally, we tackle Theorem~\ref{theorem:big-degree} in Section~\ref{section:deg-4}.
Its proof involves two very technical inequalities which are left to the Appendix~\ref{section:inequalities}.

\subsection*{Acknowledgements}
We would like to thank Ruadha\'i Dervan, Sir Simon Donaldson, Julius Ross, Yanir Rubinstein and Richard Thomas for helpful discussions and valuable comments.

We proved Theorem~\ref{theorem:main} during our visit to the Simons Center for Geometry and Physics in November 2015.
We would like to thank the center for perfect working conditions.

\section{Preliminaries}
\label{section:background}

Let $S$ be a~smooth surface, let $D$ be an effective
$\mathbb{Q}$-divisor on the surface $S$, and let $P$ be a point in
the surface $S$. Let $D=\sum_{i=1}^{r}a_iC_i$, where each $C_i$ is
an irreducible curve on $S$, and each $a_i$ is a non-negative rational number.
We assume here that all curves $C_1,\ldots,C_r$ are different.

Let $\pi\colon\widetilde{S}\to S$ be a
birational morphism such that $\widetilde{S}$ is also smooth.
Then $\pi$ is a composition of $n$ blow ups of smooth points.
For each $C_i$, denote by $\widetilde{C}_i$ its proper transform on $\widetilde{S}$.
Let $F_1,\ldots, F_n$ be $\pi$-exceptional curves.
Then
$$
K_{\widetilde{S}}+\sum_{i=1}^{r}a_i\widetilde{C}_i+\sum_{j=1}^{n}b_jF_j\sim_{\mathbb{Q}}\pi^{*}\big(K_{S}+D\big)
$$
for some rational numbers $b_1,\ldots,b_n$. Suppose, in addition,
that $\sum_{i=1}^r\widetilde{C}_i+\sum_{j=1}^n F_j$ is a divisor with
simple normal crossings (in particular each $\widetilde C_i$ is smooth).

\begin{definition}
\label{definition:lct}  The log pair $(S,D)$ is \emph{log canonical} (respectively \emph{Kawamata log terminal}) at a point $P\in S$ if
the following two conditions are satisfied:
\begin{itemize}
\item $a_i\leqslant 1$ (respectively $a_i<1$) for every $C_i$ such that $P\in C_i$,
\item $b_j\leqslant 1$ (respectively $b_j<1$) for every $F_j$ such that $\pi(F_j)=P$.
\end{itemize}
\end{definition}

This definition is independent on the choice of birational morphism $\pi\colon\widetilde{S}\to S$.
The log pair $(S,D)$ is said to be \emph{log canonical} (respectively \emph{Kawamata log terminal}) if it is log
canonical (respectively, \emph{Kawamata log terminal}) at every point of $S$.

\begin{remark}
\label{remark:convexity} Let $R$ be any effective
$\mathbb{Q}$-divisor on $S$ such that $R\sim_{\mathbb{Q}} D$ and
$R\ne D$. Let
$$
D_{\epsilon}=(1+\epsilon) D-\epsilon R,
$$
where $\epsilon\in \mathbb{Q}_{\geqslant 0}$. Then $D_{\epsilon}\sim_{\mathbb{Q}} D$.
Let $\epsilon_0$ be the greatest rational number such that $D_{\epsilon_0}$ is effective.
Then the support of  $D_{\epsilon_0}$ does not contain at least one irreducible component of $\mathrm{Supp}(R)$.
Moreover, if the log pair $(S,D)$ is not log canonical at $P$, and $(S,R)$ is log canonical at $P$, then the log pair $(S,D_{\epsilon_0})$ is not log canonical at $P$, because
$$
D=\frac{1}{1+\epsilon_0}D_{\epsilon_0}+\frac{\epsilon_0}{1+\epsilon_0}R.
$$
\end{remark}

The following result is well known and is very easy to prove.

\begin{lemma}
\label{lemma:Skoda} Suppose that $(S,D)$ is not log canonical at $P$. Then $\mathrm{mult}_{P}(D)>1$.
\end{lemma}

The following result is also well known (see \cite[Corollary~3.12]{Shokurov} or \cite[Theorem~7]{Ch13}).

\begin{lemma}
\label{lemma:adjunction}
Suppose that $(S,D)$ is not log canonical at $P$, the curve $C_1$ is smooth at $P$, and $a_1\leqslant 1$.
Let $\Delta=\sum_{i=2}^{r}a_iC_i$. Then  $\mathrm{mult}_{P}(C_1\cdot\Delta)>1$.
\end{lemma}

The following result plays an important role in the proof of Theorem~\ref{theorem:main}.

\begin{theorem}[{\cite[Theorem~13]{Ch13}}]
\label{theorem:Trento}
Suppose that $(S,D)$ is not log canonical at $P$, the curves $C_1$ and $C_2$ are smooth at $P$ and intersect each other transversally at $P$,
$a_1\leqslant 1$ and $a_2\leqslant 1$.
Let $\Delta=\sum_{i=3}^{r}a_iC_i$.
If $\mathrm{mult}_{P}(\Delta)\leqslant 1$, then $\mathrm{mult}_{P}(C_1\cdot\Delta)>2\big(1-a_{2}\big)$ or $\mathrm{mult}_{P}(C_2\cdot\Delta)>2\big(1-a_{1}\big)$.
\end{theorem}

Let $\rho\colon\widehat{S}\to S$ be the~blow up of a point $P\in S$, and let
$F$ be the $\rho$-exceptional curve. Denote by $\widehat{D}$ the~proper
transform of the divisor $D$ on the surface $\widehat{S}$ via $\rho$.
Then
$$
K_{\widehat{S}}+\widehat{D}+\Big(\mathrm{mult}_{P}(D)-1\Big)F\sim_{\mathbb{Q}}\rho^{*}\big(K_{S}+D\big).
$$
By Definition~\ref{definition:lct}, if $(S,D)$ is not log canonical
at $P$, then $(\widehat{S}, \widehat{D}+(\mathrm{mult}_{P}(D)-1)F)$ is
not log canonical at some point of the curve $F$.

\section{Del Pezzo surfaces of degree one}
\label{section:degree-one}

Let $S$ be a smooth del Pezzo surface such that $K_S^2=1$ and $L$ be an ample $\mathbb{Q}$-divisor on the surface $S$.
The goal of this section is to prove the following result:

\begin{theorem}
\label{theorem:degree-one}
If $-K_S-\frac{2}{3}\frac{-K_S\cdot L}{L^2}L$ is nef, then $\alpha(S,L)>\frac{2}{3}\frac{-K_S\cdot L}{L^2}$.
\end{theorem}

Let us prove this result. Suppose that $-K_S-\frac{2}{3}\frac{-K_S\cdot L}{L^2}L$ is nef.
We can swap $L$ with $\frac{2}{3}\frac{-K_S\cdot L}{L^2}L$ and assume that $\frac{-K_S\cdot L}{L^2}=\frac{3}{2}$. Then
\begin{equation}
\label{equation:slope-condition-1}
\Big(L-\frac{2}{3}\big(-K_S\big)\Big)\cdot L=0,
\end{equation}
and the class $-K_{S}-L$ is nef. We have to show that $\alpha(S,L)>1$.

Let $R=-K_S-L$ and $\epsilon=-K_{S}\cdot R$. Then $\epsilon>0$, because $R\not\sim_{\mathbb{Q}} 0$ by \eqref{equation:slope-condition-1}.
Let
\begin{equation}
\label{equation:gamma-1}
\gamma=\left\{%
\aligned
&\frac{6}{5}\ \mathrm{if}\ \epsilon\geqslant\frac{1}{2},\\%
&\frac{3}{3-\epsilon}\ \mathrm{if}\ \epsilon<\frac{1}{2}.\\%
\endaligned\right.%
\end{equation}
Then $\gamma>1$. We will prove that $\alpha(S,L)\geqslant\gamma$ by \emph{reductio ad absurdum}.

Suppose that $\alpha(S,L)<\gamma$. Then there exist a positive rational number $\lambda<\gamma$ and an effective $\mathbb{Q}$-divisor $D$ on $S$ such that $D\sim_{\mathbb{Q}} L$,
and $(S,\lambda D)$ is not log canonical at some point~$P\in S$.

\begin{lemma}
\label{lemma:degree-one-1}
One has $L\not\sim_{\mathbb{Q}}-\frac{2}{3}K_{S}$.
\end{lemma}

\begin{proof}
By \cite[Theorem~1.7]{Ch08}, $\alpha(S, -K_S)\geqslant\frac{5}{6}$.
Thus, if $L\sim_{\mathbb{Q}}-\frac{2}{3}K_{S}$, then  $(S,\frac{5}{4}D)$ is log canonical,
which is impossible, because $\lambda<\gamma\leqslant\frac{6}{5}<\frac{5}{4}$ by \eqref{equation:gamma-1}.
\end{proof}

Let $C$ be a curve in the pencil $|-K_S|$ that passes through the point $P$. Then $C$ is irreducible.
Write $D=aC+\Delta$, where $a$ is a non-negative rational number,
and $\Delta$ is an effective $\mathbb{Q}$-divisor whose support does not contain the curve $C$.

\begin{lemma}
\label{lemma:degree-one-3}
One has $a<\frac{2}{3}$.
\end{lemma}

\begin{proof}
Suppose that $a\geqslant\frac{2}{3}$. Then it follows from \eqref{equation:slope-condition-1} that
$$
0=L\cdot\Big(L-\frac{2}{3}\big(-K_S\big)\Big)=L\cdot\Big(D-\frac{2}{3} C\Big)=L\cdot\Bigg(\Big(a-\frac{2}{3}\Big)C+\Delta\Bigg)\geqslant\Big(a-\frac{2}{3}\Big)L\cdot C,
$$
so that $a=\frac{2}{3}$ and $L\cdot\Delta=0$, because $L\cdot C>0$ and $L\cdot\Delta\geqslant 0$.
Then $\Delta=0$, because $L$ is ample. Thus, we have $L\sim_{\mathbb{Q}}-\frac{2}{3}K_{S}$,
which is impossible by Lemma~\ref{lemma:degree-one-1}.
\end{proof}

If $C$ is smooth at $P$, then it follows from Lemma~\ref{lemma:adjunction} that
$$
\frac{1}{\lambda}<C\cdot\Delta=C\cdot (D-aC)=1-R\cdot C-a\leqslant 1-R\cdot C=1-\epsilon,
$$
which gives $\epsilon<1$ and $\gamma>\lambda>\frac{1}{1-\epsilon}>\frac{3}{3-\epsilon}$, which contradicts \eqref{equation:gamma-1}. Thus, $\mathrm{mult}_P(C)=2$.

Let $m=\mathrm{mult}_P(\Delta)$. Then $2m\leqslant C\cdot\Delta=1-\epsilon-a<1-a$.

\begin{corollary}
\label{corollary:degree-one-4}
One has $a+2m<1$.
\end{corollary}

Let $f\colon \widetilde{S}\rightarrow S$ be the blow-up at $P$, and let $F$ be its exceptional curve.
Denote by $\widetilde{C}$ and $\widetilde{\Delta}$ the proper transforms on the surface $\widetilde{S}$ of the curve $C$ and the divisor $\Delta$, respectively.
Then the log pair $(\widetilde{S}, \lambda a\widetilde{C} + \lambda\widetilde{\Delta}+(\lambda(2a+m)-1)F)$ is not log canonical
at some point $Q\in F$.
Since $a<\frac{2}{3}$ by Lemma~\ref{lemma:degree-one-3} and $a+2m<1$ by Corollary~\ref{corollary:degree-one-4}, we have $2a+m<\frac{3}{2}$.
Then $\lambda(2a+m)-1<1$, because $\lambda<\frac{6}{5}$ by \eqref{equation:gamma-1}.
We also have $\lambda(2a+m)-1>0$ by Lemma~\ref{lemma:Skoda}.

\begin{lemma}
\label{lemma:degree-one-5}
The point $Q$ is contained in the curve $\widetilde{C}$.
\end{lemma}

\begin{proof}
If $Q\not\in\widetilde{C}$, then $\frac{1}{2}\geqslant\frac{1-a}{2}\geqslant m=F\cdot\widetilde{\Delta}>\frac{1}{\lambda}>\frac{5}{6}$
by Lemma~\ref{lemma:adjunction}, Corollary~\ref{corollary:degree-one-4} and \eqref{equation:gamma-1}.
\end{proof}

By Lemma~\ref{lemma:degree-one-3} and \eqref{equation:gamma-1}, we have $\lambda a<1$. Observe that $\widetilde C \cdot F =2$ and the curve $\widetilde{C}$ is smooth. Thus, applying Lemma~\ref{lemma:adjunction}, we get
$$
1<\widetilde{C}\cdot\Big(\lambda\widetilde{\Delta}+\big(\lambda(2a+m)-1\big)F\Big)=\lambda(C\cdot\Delta-2m)+2\lambda(2a+m)-2=\lambda(1-\epsilon +3a)-2.
$$
Now using $\epsilon>0$ and $a<\frac{2}{3}$, we obtain
$3<\lambda(1+3a-\epsilon)<\lambda\left(3-\epsilon\right)<3$
by \eqref{equation:gamma-1}.

The obtained contradiction completes the proof of Theorem~\ref{theorem:degree-one}.

\section{Del Pezzo surfaces of degree two}
\label{section:degree-two}

Let $S$ be a smooth del Pezzo surface such that $K_S^2=2$ and $L$ be an ample $\mathbb{Q}$-divisor on the surface $S$.
In this section we prove the following result:

\begin{theorem}
\label{theorem:degree-two}
If $-K_S-\frac{2}{3}\frac{-K_S\cdot L}{L^2}L$ is nef, then $\alpha(S,L)>\frac{2}{3}\frac{-K_S\cdot L}{L^2}$.
\end{theorem}

As in the proof of Theorem \ref{theorem:degree-one}, we may assume that $\frac{-K_S\cdot L}{L^2}=\frac{3}{2}$. Then
\begin{equation}
\label{equation:fixed-slope-condition-2}
\Big(L-\frac{2}{3}\big(-K_S\big)\Big)\cdot L=0.
\end{equation}
Suppose that $-K_{S}-L$ is nef. To prove  Theorem~\ref{theorem:degree-two}, we have to show that $\alpha(S,L)>1$.

Let $R=-K_S-L$ and $\epsilon=-K_{S}\cdot R$. Then $\epsilon>0$, because $R\not\sim_{\mathbb{Q}} 0$ by \eqref{equation:fixed-slope-condition-2}.
Let
\begin{equation}
\label{equation:gamma-2}
\gamma=\left\{%
\aligned
&\frac{12}{11}\ \mathrm{if}\ \epsilon\geqslant 1,\\%
&\frac{12}{12-\epsilon}\ \mathrm{if}\ \epsilon<1.\\%
\endaligned\right.%
\end{equation}
Then $\gamma>1$. We will prove that $\alpha(S,L)\geqslant\gamma$ by \emph{reductio ad absurdum}.

Suppose that $\alpha(S,L)<\gamma$. Then there exist a positive rational number $\lambda<\gamma$ and an effective $\mathbb{Q}$-divisor $D$ on $S$ such that $D\sim_{\mathbb{Q}} L$, and $(S,\lambda D)$ is not log canonical at some point~$P\in S$.

\begin{lemma}
\label{lemma:degree-two-1}
One has $L\not\sim_{\mathbb{Q}}-\frac{2}{3}K_{S}$.
\end{lemma}

\begin{proof}
By \cite[Theorem~1.7]{Ch08} (c.f. \cite{MarGar}), $\alpha(S, -K_S)\geqslant\frac{3}{4}$.
Thus, if $L\sim_{\mathbb{Q}}-\frac{2}{3}K_{S}$, then $(S,\frac{9}{8}D)$ is log canonical,
which is impossible, because $\lambda<\gamma\leqslant\frac{12}{11}<\frac{9}{8}$ by \eqref{equation:gamma-2}.
\end{proof}

\begin{lemma}
\label{lemma:degree-two-2}
Suppose that $D=\sum_{i=1}^k a_iC_i+\Delta$, where $a_1,\ldots,a_k$ are non-negative rational numbers,
$C_1,\ldots,C_k$ are irreducible curves on the surface $S$ such that 
$$
\sum_{i=1}^kC_i\sim -K_{S},
$$
and $\Delta$ is an effective $\mathbb{Q}$-divisor on the surface $S$ whose support does not contain any curve among the curves $C_1,\ldots,C_k$.
Then $a_r<\frac{2}{3}$ for some $r$.
\end{lemma}

\begin{proof}
Suppose that $a_1\geqslant\frac{2}{3},\ldots,a_k\geqslant \frac{2}{3}$.
Then it follows from \eqref{equation:fixed-slope-condition-2} that
$$
0=L\cdot\Big(L-\frac{2}{3}\big(-K_S\big)\Big)=\sum_{i=1}^k\Big(a_i-\frac{2}{3}\Big)L\cdot C_i+L\cdot\Delta\geqslant \sum_{i=1}^k\Big(a_i-\frac{2}{3}\Big)L\cdot C_i\geqslant 0,
$$
which implies that $a_1=\cdots =a_k=\frac{2}{3}$ and $L\cdot\Delta=0$. Then $\Delta=0$, since the divisor $L$ is ample and $L\sim_{\mathbb{Q}}-\frac{2}{3}K_{S}$, which contradicts Lemma~\ref{lemma:degree-two-1}.
\end{proof}

Let $f\colon \widetilde{S}\rightarrow S$ be a blow-up of the point $P$, let $F$ be its exceptional curve,
and let $\widetilde{D}$ be the proper transform on $\widetilde{S}$ of the divisor $D$.
Then $(\widetilde{S}, \lambda\widetilde{D}+(\lambda\mathrm{mult}_P(D)-1)F)$ is not log canonical at some point $Q\in F$.
Note that $\lambda\mathrm{mult}_P(D)>1$ by Lemma~\ref{lemma:Skoda}.

\begin{lemma}
\label{lemma:degree-two-3a}
One has $\mathrm{mult}_P(D)\leqslant 2-\epsilon$ and $\lambda\mathrm{mult}_P(D)<2$.
\end{lemma}

\begin{proof}
Let $C$ be a general curve in $|-K_S|$ that passes through $P$. Then $\mathrm{mult}_P(D)\leqslant D\cdot C=2-\epsilon$,
which implies that $\lambda\mathrm{mult}_P(D)<2$ by \eqref{equation:gamma-2}.
\end{proof}

Let $g\colon\widehat{S}\rightarrow \widetilde{S}$ be the blow-up at the point $Q$, and let $G$ be its exceptional curve.
Denote by $\widehat{F}$ and $\widehat{D}$ the proper transforms on $\widehat{S}$ of the curve $F$ and the divisor $\widetilde{D}$, respectively.
Then the log pair
$(\widehat{S},\lambda\widehat{D}+(\lambda\mathrm{mult}_P(D)-1)\widehat{F}+(\lambda(\mathrm{mult}_P(D)+\mathrm{mult}_Q(\widetilde{D}))-2)G)$
is not log canonical at some point $O\in G$. Note that it follows from Lemma~\ref{lemma:Skoda} that
$\lambda\mathrm{mult}_P(D)+\lambda\mathrm{mult}_Q(\widetilde{D})>2$.

The linear system $|-K_{S}|$ gives a double cover $\pi\colon S\rightarrow \mathbb P^2$
branched over a smooth quartic curve $C_4\subset\mathbb P^2$.
Thus, if $\pi(P)\in C_4$, then $|-K_S|$ contains a unique curve that is singular at $P$.

\begin{lemma}
\label{lemma:degree-two-3}
Suppose that $\pi(P)$ is contained in the curve $C_4$.
Let $T$ be the unique curve in the linear system $|-K_S|$ that is singular at the point $P$.
Then $T$ is reducible.
\end{lemma}

\begin{proof}
Suppose that $T$ is irreducible. 
Write $D=aT+\Delta$, where $a$ is a non-negative rational number,
and $\Delta$ is an effective $\mathbb{Q}$-divisor whose support does not contain $T$.
By Lemma \ref{lemma:degree-two-2}, we have $a<\frac{2}{3}$. Let $m=\mathrm{mult}_P(\Delta)$. Then
\begin{equation}
\label{equation:deg2-bound1}
2m\leqslant T\cdot\Delta=2-\epsilon-2a\leqslant 2-\epsilon.
\end{equation}
Denote by $\widetilde{T}$ and $\widetilde{\Delta}$ the proper transforms on $\widetilde{S}$ of the curve $T$ and the divisor $\Delta$, respectively.
If $Q\not\in\widetilde{T}$, then Lemma~\ref{lemma:adjunction} gives
$\frac{1}{\lambda}<F\cdot\widetilde{\Delta}=m\leqslant 1-\frac{\epsilon}{2}$,
which contradicts \eqref{equation:gamma-2}. Then $Q\in \widetilde{T}$.

Since $T$ is irreducible, it has either a nodal point at $P$ or a cuspidal point.
In particular, the curve $\widetilde{T}$ is smooth.
If $T$ has a nodal point at $P$, then Lemma~\ref{lemma:adjunction} gives
\begin{equation}
2-\epsilon-2a-2m=\widetilde{T}\cdot\widetilde{\Delta}\geqslant\mathrm{mult}_{Q}\Big(\widetilde{T}\cdot\widetilde{\Delta}\Big)>\frac{2}{\lambda}-2a-m,
\label{equation:deg2-bound1-2}
\end{equation}
so that $2-\epsilon\geqslant 2-\epsilon-m>\frac{2}{\lambda}$, which contradicts \eqref{equation:gamma-2}.
Thus, $T$ has a cuspidal singularity at $P$.

Denote by $\widehat{T}$ and $\widehat{\Delta}$ the proper transforms
on $\widehat{S}$ of the curve $\widetilde{T}$ and the divisor $\widetilde{\Delta}$, respectively.
Let $\widetilde{m}=\mathrm{mult}_{Q}(\widetilde{\Delta})$.
Then
$(\widehat{S},\lambda a\widehat{T}+ \lambda\widehat{\Delta}+(\lambda(2a+m)-1)\widehat{F}+(\lambda(3a+m+\widetilde{m})-2)G)$
is not log canonical at $O$.
Using Lemma~\ref{lemma:degree-two-3a}, \eqref{equation:gamma-2} and \eqref{equation:deg2-bound1-2}, we get $\lambda(3a+m+\widetilde{m})\leqslant\lambda(\frac{8}{3}-\epsilon)<3$,
since $a<\frac{2}{3}$.

If $O\ne\widehat{T}\cap\widehat{F}\cap G$, then 
$1-\frac{\epsilon}{2}<\frac{1}{\lambda}<G\cdot\widehat{\Delta}=\widetilde{m}\leqslant m$ by Lemma~\ref{lemma:adjunction} and \eqref{equation:gamma-2},
which is absurd, because $m\leqslant 1-\frac{\epsilon}{2}$ by \eqref{equation:deg2-bound1}.
Thus, we have $O=\widehat{T}\cap\widehat{F}\cap G$. Then
$$
1<\widehat{T}\cdot\Big(\lambda\widehat{\Delta}+\big(\lambda(2a+m)-1\big)F+\big(\lambda(3a+m+\widetilde{m})-2\big)G\Big)=\lambda(2+3a-\epsilon)-3
$$
by Lemma~\ref{lemma:adjunction}. Since $a<\frac{2}{3}$, this gives
$1<\lambda(4-\epsilon)-3$,
which is impossible by \eqref{equation:gamma-2}.
\end{proof}

\begin{lemma}
\label{lemma:degree-two-4}
The point $\pi(P)$ is not contained in the curve $C_4$.
\end{lemma}

\begin{proof}
Suppose that $\pi(P)\in C_4$.
Let $T$ be the curve in $|-K_S|$ that is singular at the point $P$.
By Lemma~\ref{lemma:degree-two-3}, it is reducible,
so that $T=L_1+L_2$, where $L_1$ and $L_2$ are irreducible smooth curves such that $L_1^2=L_2^2=-1$, $L_1\cdot L_2=2$ and $P\in L_1\cap L_2$.
Let $\epsilon_1=L_1\cdot R$ and $\epsilon_2=L_2\cdot R$. Then $\epsilon=\epsilon_1+\epsilon_2$,
where $\epsilon_1\geqslant 0$ and $\epsilon_2\geqslant 0$, since $R$ is nef.

Write $D=a_1L_1+a_2L_2+\Delta$, where $a_1$ and $a_2$ are non-negative rational numbers,
and $\Delta$ is an effective $\mathbb{Q}$-divisor whose support contains neither $L_1$ nor $L_2$.
Then $a_1<\frac{2}{3}$ or $a_2<\frac{2}{3}$ by Lemma~\ref{lemma:degree-two-2}.
Without loss of generality, we may assume that $a_1<\frac{2}{3}$. It follows from \eqref{equation:fixed-slope-condition-2} that
\begin{multline*}
0=L\cdot\Big(D-\frac{2}{3}\big(L_1+L_2\big)\Big)=L\cdot\Bigg(\Big(a_1-\frac{2}{3}\Big)L_1+\Big(a_2-\frac{2}{3}\Big)L_2+\Delta\Bigg)=\\
=\Big(a_1-\frac{2}{3}\Big)\big(1-\epsilon_1\big)+\Big(a_2-\frac{2}{3}\Big)\big(1-\epsilon_2\big)+\Delta\cdot L\geqslant\Big(a_1-\frac{2}{3}\Big)\big(1-\epsilon_1\big)+\Big(a_2-\frac{2}{3}\Big)\big(1-\epsilon_2\big).
\end{multline*}
This gives us
\begin{equation}
\label{equation:extra}
a_1+a_2\leqslant \frac{4}{3}-\frac{2}{3}\epsilon+a_1\epsilon_1+a_2\epsilon_2.
\end{equation}

Let $m=\mathrm{mult}_{P}(\Delta)$. Then we get two inequalities:
\begin{equation}
\label{equation:deg2-bound2-1}
1-\epsilon_1+a_1-2a_2=\Delta\cdot L_1\geqslant m,
\end{equation}
and $1-\epsilon_2-2a_1+a_2=\Delta\cdot L_2\geqslant m$.
Adding them, we get $2-\epsilon\geqslant a_1+a_2+2m$,
so that $m\leqslant 1-\frac{\epsilon}{2}$.
Since $a_1<\frac{2}{3}$, it follows from \eqref{equation:deg2-bound2-1} that
$2a_2\leqslant 1-\epsilon_1+a_1\leqslant 1+a_1<\frac{5}{3}$,
so that $a_2<\frac{5}{6}$.

Denote by $\widetilde{L}_1$, $\widetilde{L}_2$ and $\widetilde \Delta$  the proper transforms on $\widetilde S$ of $L_1$, $L_2$  and $\Delta$, respectively.
Then the log pair
$(\widetilde{S},\lambda a_1\widetilde{L}_1+\lambda a_2\widetilde{L}_2+\lambda\widetilde{\Delta}+\big(\lambda(a_1+a_2+m)-1)F)$
is not log canonical at the point $Q$.
If~$Q\not\in\widetilde{L}_1\cup\widetilde{L}_2$, then Lemma~\ref{lemma:adjunction}, Lemma \ref{lemma:degree-two-3a} and \eqref{equation:gamma-2} give
$m=\widetilde{\Delta}\cdot F>\frac{1}{\lambda}>\frac{1}{\gamma}\geqslant\frac{2-\epsilon}{2}$,
which is impossible, since $m\leqslant 1-\frac{\epsilon}{2}$.
Thus, we have $Q\in\widetilde{L}_1\cup\widetilde{L}_2$.
If $Q\not\in\widetilde{L}_1$, then Lemma~\ref{lemma:adjunction} gives
$$
\frac{2}{\lambda}-\big(a_1+a_2+m\big)<\widetilde{L}_2\cdot\widetilde{\Delta}=1-\epsilon_2-2a_1+a_2-m,
$$
so that $1-\epsilon_2+2a_2-a_1>\frac{2}{\lambda}$.
Adding this inequality to \eqref{equation:deg2-bound2-1}, we get
$2-\epsilon>\frac{2}{\lambda}+m\geqslant\frac{2}{\lambda}$,
which contradicts \eqref{equation:gamma-2}. Thus, we see that $Q\in\widetilde{L}_1$. Similarly, it can be shown that $Q\in\widetilde{L}_2$.

We have $Q=F\cap\widetilde{L}_1\cap\widetilde{L}_2$. Then $L_1$ is tangent to $L_2$ at the point $P$.
Let $\widetilde{m}=\mathrm{mult}_{Q}(\widetilde{\Delta})$. Then
$\widetilde{\Delta}\cdot\widetilde{L}_1\geqslant \widetilde{m}$ and $\widetilde{\Delta}\cdot\widetilde{L}_2\geqslant\widetilde{m}$.
This gives
\begin{equation}
\label{equation:deg2-bound4-1}
\left\{%
\aligned
&1-\epsilon_1+a_1-2a_2-m\geqslant \widetilde{m},\\%
&1-\epsilon_2-2a_1+a_2-m\geqslant\widetilde{m}.\\%
\endaligned\right.
\end{equation}
In particular, adding the two inequalities, we obtain
\begin{equation}
\label{equation:deg2-bound5}
2-\epsilon\geqslant a_1+a_2+2m+2\widetilde{m}.
\end{equation}

Denote by $\widehat{L}_1$, $\widehat{L}_2$, $\widehat{F}$  and $\widehat{\Delta}$ the proper transforms of $\widetilde{L}_1$, $\widetilde{L}_2$, $F$ and $\widetilde{\Delta}$ on $\widehat{S}$, respectively.
Then
$(\widehat{S},\lambda a_1\widehat{L}_1+\lambda a_2\widehat{L}_2+\lambda\widehat{\Delta}+(\lambda(a_1+a_2+m)-1)\widehat{F}+(\lambda(2a_1+2a_2+m+\widetilde{m})-2\big)G)$
is not log canonical at the point $O$. Moreover, we have
$2a_1+2a_2+m+\widetilde{m}\leqslant 1-\frac{\epsilon}{2}+\frac{3}{2}(a_1+a_2)$ by~\eqref{equation:deg2-bound5}.
Now using \eqref{equation:gamma-2}, \eqref{equation:extra}, $a_1<\frac{2}{3}$ and $a_2<\frac{5}{6}$, we get
\begin{multline*}
2a_1+2a_2+m+\widetilde{m}\leqslant 1-\frac{\epsilon}{2}+\frac{3}{2}\Big(\frac{4}{3}-\frac{2}{3}\epsilon+a_1\epsilon_1+a_2\epsilon_2\Big)=3-\frac{3}{2}\Big(\epsilon-a_1\epsilon_1-a_2\epsilon_2\Big)=\\
=3+\frac{3}{2}\Big(a_1-1\Big)\epsilon_1+\frac{3}{2}\Big(a_2-1\Big)\epsilon_2<3-\frac{1}{2}\epsilon_1-\frac{1}{4}\epsilon_2\leqslant 3-\frac{1}{4}\epsilon_1-\frac{1}{4}\epsilon_2=3-\frac{1}{4}\epsilon=\frac{12-\epsilon}{4}\leqslant\frac{3}{\gamma}<\frac{3}{\lambda}.
\end{multline*}

If $O\not\in\widehat{L}_1\cup\widetilde{L}_2\cup\widehat{F}$, then Lemma~\ref{lemma:adjunction} and \eqref{equation:gamma-2} give
$\widetilde{m}=\widehat{\Delta}\cdot G>\frac{1}{\lambda}>\frac{2-\epsilon}{2}$,
which is impossible, because $\widetilde{m}\leqslant m\leqslant 1-\frac{\epsilon}{2}$.
If $O=\widehat{F}\cap G$, then Lemma~\ref{lemma:adjunction} gives
$\frac{2}{\lambda}-(a_1+a_2+m)<G\cdot\widehat{\Delta}=\widetilde{m}$,
so that $a_1+a_2+m+\widetilde{m}>\frac{2}{\lambda}$.
This together with \eqref{equation:deg2-bound5} imply
$2-\epsilon\geqslant a_1+a_2+2m+2\widetilde{m}\geqslant a_1+a_2+m+\widetilde{m}>\frac{2}{\lambda}$,
which contradicts \eqref{equation:gamma-2}. Thus, we see that either $O=\widehat{L}_1\cap G$ or $O=\widehat{L}_2\cap G$.

Suppose that $O=\widehat{L}_1\cap G$. Then $\widetilde{m}=G\cdot\widehat{\Delta}>\frac{1}{\lambda}-a_1$ by Lemma~\ref{lemma:adjunction}.
Plugging this into \eqref{equation:deg2-bound4-1}, we obtain
$1-\epsilon_2-2a_1+a_2\geqslant m+\widetilde{m}\geqslant 2\widetilde{m}>\frac{2}{\lambda}-2a_1$,
so that $1-\epsilon_2+a_2>\frac{2}{\lambda}$. But
$$
\frac{3}{\lambda}-\big(2a_1+2a_2+m+\widetilde{m}\big)<\widehat{L}_1\cdot\widehat{\Delta}=1-\epsilon_1+a_1-2a_2-m-\widetilde{m}
$$
by Lemma~\ref{lemma:adjunction}, so that $1-\epsilon_1+3a_1>\frac{3}{\lambda}$. Adding these inequalities, we get
$3a_1+a_2+2-\epsilon>\frac{5}{\lambda}$.
Now using $a_1<\frac{2}{3}$, $a_2<\frac{5}{6}$ and \eqref{equation:extra}, we get
\begin{multline*}
\frac{5}{\lambda}<3a_1+a_2+2-\epsilon<\frac{10}{3}+a_1+a_2-\epsilon\leqslant \frac{10}{3}+\Big(\frac{4}{3}-\frac{2}{3}\epsilon+a_1\epsilon_1+a_2\epsilon_2\Big)-\epsilon=\\
=\frac{14}{3}-\frac{5}{3}\epsilon+a_1\epsilon_1+a_2\epsilon_2<\frac{14}{3}-\frac{5}{3}\epsilon+\frac{2}{3}\epsilon_1+\frac{5}{6}\epsilon_2=\frac{14}{3}-\epsilon_1-\frac{5}{6}\epsilon_2\leqslant\frac{14}{3}-\frac{5}{6}\epsilon,
\end{multline*}
which implies that $\gamma>\lambda>\frac{30}{28-5\epsilon}$. This is impossible by \eqref{equation:gamma-2}.
We conclude that $O\ne\widehat{L}_1\cap G$.

We have $O=\widehat{L}_2\cap G$.
By \eqref{equation:deg2-bound5}, we also have
$4\widetilde{m}\leqslant 2m+2\widetilde{m}\leqslant a_1+a_2+2m+2\widetilde{m}\leqslant 2-\epsilon$,
so that $\widetilde{m}\leqslant\frac{1}{2}$.
Then \eqref{equation:gamma-2} gives
$\mathrm{mult}_{O}(\lambda\widehat{\Delta})\leqslant\lambda\widetilde{m}\leqslant\frac{\lambda}{2}<1$.
Thus, applying Theorem \ref{theorem:Trento} either we obtain
$2(\frac{1}{\lambda}-a_2)<\widehat{\Delta}\cdot G=\widetilde{m}$,
or we obtain
\begin{equation}
\label{equation:long-one}
2\Big(\frac{3}{\lambda}-\big(2(a_1+a_2)+m+\widetilde{m}\big)\Big)<\widehat{\Delta}\cdot\widehat{L}_2=1-\epsilon_2-2a_1+a_2-m-\widetilde{m}.
\end{equation}
If $\widetilde{m}>2(\frac{1}{\lambda}-a_2)$, then \eqref{equation:gamma-2} and \eqref{equation:deg2-bound4-1} imply
$$
1+a_1-2a_2\geqslant 1-\epsilon_1+a_1-2a_2\geqslant m+\widetilde{m}\geqslant 2\widetilde{m}>4\Big(\frac{1}{\lambda}-a_2\Big)=\frac{4}{\lambda}-4a_2\geqslant\frac{11}{3}-4a_2,
$$
which gives $a_1+2a_2>\frac{8}{3}$.
But $a_1+2a_2<\frac{7}{3}$,
since $a_1<\frac{2}{3}$ and $a_2<\frac{5}{6}$.
We see that \eqref{equation:long-one} holds.
It gives us $1-\epsilon_2+2a_1+5a_2+m+\widetilde{m}>\frac{6}{\lambda}$.
But $m+\widetilde{m}\leqslant 1-\epsilon_1+a_1-2a_2$ by \eqref{equation:deg2-bound4-1},
so that we obtain $2-\epsilon+3(a_1+a_2)\geqslant 1-\epsilon_2+2a_1+5a_2+m+\widetilde{m}>\frac{6}{\lambda}$.
Now using \eqref{equation:extra}, we get
$$
\frac{6}{\lambda}<2-\epsilon+3(a_1+a_2)\leqslant 6-3\epsilon+3a_1\epsilon_1+3a_2\epsilon_2<6-3\epsilon+2\epsilon_1+\frac{5}{2}\epsilon_2\leqslant 6-\frac{\epsilon}{2},
$$
because $a_1<\frac{2}{3}$ and $a_2<\frac{5}{6}$.
Thus, we see that $\lambda>\frac{12}{12-\epsilon}$, which is impossible by \eqref{equation:gamma-2}.
\end{proof}

Let $\widetilde{C}$ be a curve in the pencil $|-K_{\widetilde{S}}|$ that passes through $Q$.
Denote by $C$ its image on the surface $S$. Then $P\in C$ and $C\in|-K_{S}|$.
Moreover, the curve $C$ is smooth at $P$, because $P$ is not contained in $C_4$ by Lemma~\ref{lemma:degree-two-4}.
Furthermore, we have

\begin{lemma}
\label{lemma:degree-two-5}
The curve $C$ is reducible.
\end{lemma}

\begin{proof}
Suppose that $C$ is irreducible.
Let us write $D=aC+\Delta$, where $a$ is a non-negative rational number,
and $\Delta$ is an effective $\mathbb{Q}$-divisor on the surface $S$,
whose support does not contain the curve $C$.
Then $a<\frac{2}{3}$ by Lemma \ref{lemma:degree-two-2}.

Denote by $\widetilde{\Delta}$ the proper transform of the divisor $\Delta$ on the surface $\widetilde{S}$.
Let $m=\mathrm{mult}_{P}(\Delta)$. Then
$(\widetilde{S}, \lambda a\widetilde{C}+\lambda\widetilde{\Delta}+(\lambda (a+m)-1)F)$ is not log canonical at $Q$.
Now, applying Lemma~\ref{lemma:adjunction}, we get $\frac{2}{\lambda}-(a+m)<\widetilde{C}\cdot\widetilde{\Delta}=2-2a-\epsilon-m$,
so that $\frac{2}{\lambda}<2-\epsilon$, which contradicts \eqref{equation:gamma-2}.
\end{proof}

Thus, we have $C=L_1+L_2$, where $L_1$ and $L_2$ are irreducible curves such that $L_1^2=L_2^2=-1$ and $L_1\cdot L_2=2$.
Since $C$ is smooth at $P$, we have $P\not\in L_1\cap L_2$.
Without loss of generality, we may assume that $P\in L_1$.
Let $\epsilon_1=L_1\cdot R$ and $\epsilon_2=L_2\cdot R$. Then $\epsilon=\epsilon_1+\epsilon_2$.

Write $D=a_1L_1+a_2L_2+\Delta$, where $a_1$ and $a_2$ are non-negative rational numbers,
and $\Delta$ is an effective $\mathbb{Q}$-divisor, whose support contains neither $L_1$ nor $L_2$.
Then $a_1\leqslant \frac{1+a_2}{2}$, since
\begin{equation}
\label{equation:degree-two-last}
1-\epsilon_2-2a_1+a_2=L_2\cdot\Delta\geqslant 0.
\end{equation}
If $a_2<\frac{2}{3}$, then
$a_1\leqslant \frac{1+a_2}{2}<\frac{5}{6}$.
Vice versa, if $a_2\geqslant\frac{2}{3}$, then $a_1<\frac{2}{3}$ by Lemma~\ref{lemma:degree-two-2}.
In both cases we have $a_1<\frac{5}{6}$.
Then $\lambda a_1\leqslant\frac{10}{11}$ by \eqref{equation:gamma-2}.

Denote by $\widetilde{L}_1$, $\widetilde{L}_2$ and $\widetilde \Delta$ the proper transforms on $\widetilde S$ of $L_1$, $L_2$ and $\Delta$, respectively.
Then the log pair
$(\widetilde{S},\lambda a_1\widetilde{L}_1+\lambda\widetilde{\Delta}+(\lambda(a_1+m)-1)F)$
is not log canonical at $Q=\widetilde L_1\cap F$, by construction. Since $\lambda a_1\leqslant 1$, we can apply Lemma~\ref{lemma:adjunction} to this log pair and the curve $\widetilde{L}_1$.
This gives
$$
\frac{2}{\lambda}-(a_1+m)<\widetilde{L}_1\cdot\widetilde{\Delta}=1-\epsilon_1+a_1-2a_2-m,
$$
so that $\frac{2}{\lambda}<1-\epsilon_1+2a_1-2a_2$. Using \eqref{equation:degree-two-last}, we get
$\frac{2}{\lambda}<2-\epsilon_1-\epsilon_2-a_2\leqslant 2-\epsilon_1-\epsilon_2=2-\epsilon$,
which implies that $\lambda>\frac{2}{2-\epsilon}$. This is impossible by \eqref{equation:gamma-2}.

The obtained contradiction completes the proof of Theorem~\ref{theorem:degree-two}.

\section{Cubic surfaces}
\label{section:cubic-surfaces}

Let $S$ be a smooth cubic surface in $\mathbb{P}^3$.
Fix six disjoint lines $E_1$, $E_2$, $E_3$, $E_4$, $E_5$, $E_6$ in $S$, and fix a positive rational number $x\leqslant\frac{1}{10}$.
Let $L=-K_{S}+x\sum_{i=1}^6E_i$. Then $L$ is ample.

The goal of this section is to prove Theorem~\ref{theorem:cubic-surface-1}.
By Theorem~\ref{theorem:Dervan}, to do this it is enough to show that $\alpha(S,L)\geqslant\frac{2}{3+3x}$.
Suppose that this is not true.
Then there exist an effective $\mathbb{Q}$-divisor $D$ on the surface $S$ and a positive rational number $\lambda<\frac{2}{3+3x}$ such that
$D\sim_{\mathbb{Q}} L$, and $(S,\lambda D)$ is not log canonical at some point $P\in S$.
In this section we seek for a contradiction.

\begin{lemma}
\label{lemma:cubic-surface-codimension-one}
The log pair $(S,\lambda D)$ is log canonical outside of finitely many points.
\end{lemma}

\begin{proof}
Suppose that this is not true.
Then $D=aC+\Delta$, where $C$ is an irreducible curve,
$a$ is a positive rational number such that $a>\frac{1}{\lambda}$,
and $\Delta$ is an effective $\mathbb{Q}$-divisor,
whose support does not contain the curve $C$.
Denote by $d$ the degree of the curve $C$ in $\mathbb{P}^3$. Then
$$
3+6x=-K_{S}\cdot D=-K_S\cdot\big(aC+\Delta\big)=ad-K_{S}\cdot\Delta\geqslant ad>\frac{d}{\lambda}>\frac{3d+3xd}{2}.
$$
which implies that $d\leqslant 2$. Then $C$ is a line or a conic.
Denote by $Z$ a general curve in $|-K_S-C|$.

Suppose that $C$ is a line.
Then $Z$ is a smooth conic, $C\cdot Z=2$, and $Z$ is not contained in the support of $\Delta$.
If $C$ is one of the lines $E_1$, $E_2$, $E_3$, $E_4$, $E_5$, $E_6$, then
$2+7x-2a\geqslant\Delta\cdot Z\geqslant 0$,
so that $3+3x<\frac{2}{\lambda}<2a\leqslant 2+7x$, which is impossible, because $x\leqslant\frac{1}{10}$. Then
$$
2+6x-2a\geqslant\Big(-K_{S}+x\sum_{i=1}^6E_i-aC\Big)\cdot Z=(L-aC)\cdot Z=\Delta\cdot Z\geqslant 0,
$$
so that $2a\leqslant 2+6x$, which is impossible, because $2a>\frac{2}{\lambda}>3+3x$ and $x\leqslant\frac{1}{10}$.

We see that $C$ is a conic. Then $Z$ is a line.
Write $\Delta=bZ+\Omega$, where $b$ is a non-negative rational number,
and $\Omega$ is an effective $\mathbb{Q}$-divisor,
whose support contains no $Z$.
If $Z$ is one of the lines $E_1$, $E_2$, $E_3$, $E_4$, $E_5$, $E_6$, then
$1-x-2a+b=Z\cdot\Omega\geqslant 0$ and
$2+7x-2b\geqslant C\cdot\Omega\geqslant 0$,
so that $4+5x\geqslant 4a>\frac{4}{\lambda}>6+6x$, which is absurd.
Observe that $(C+Z)\cdot E_i=1$ for each $E_i$.
But
$$
1+x\sum_{i=1}^6 Z\cdot E_i-2a+b=Z\cdot\Omega\geqslant 0
$$
and $2+x \sum_{i=1}^6 C\cdot E_i-2b=C\cdot\Omega\geqslant 0$. This gives
$$
4+12x\geqslant 4+x\Big(6+\sum_{i=1}^6 Z\cdot E_i\Big)=4+x\Big(2\sum_{i=1}^6 Z\cdot E_i+\sum_{i=1}^6 C\cdot E_i\Big)\geqslant 4a>\frac{4}{\lambda}>6+6x,
$$
so that $x>\frac{1}{3}$, which is impossible, because $x\leqslant\frac{1}{10}$.
\end{proof}

Let $h\colon S\to\mathbb{P}^2$ be the contraction of the lines $E_1$, $E_2$, $E_3$, $E_4$, $E_5$, $E_6$,

\begin{lemma}
\label{lemma:cubic-surface-A-P-in-E1-E2-E3-E4-E5-E6}
The point $P$ is contained in $E_1\cup E_2\cup E_3\cup E_4\cup E_5\cup E_6$.
\end{lemma}

\begin{proof}
Suppose that $P\not\in E_1\cup E_2\cup E_3\cup E_4\cup E_5\cup E_6$.
and let $\ell$ be a line in $\mathbb{P}^2$ such $h(P)\not\in\ell$.
Let $\overline{D}=h(D)$, and denote by $\mathrm{LCS}(\mathbb{P}^2, \lambda\overline{D}+\ell)$ the locus in $\mathbb{P}^2$ consisting of all points
where the log pair $(\mathbb{P}^2, \lambda\overline{D}+\ell)$ is not Kawamata log terminal.
Then this set contains both $h(P)$ and $\ell$, 
so that it is not connected by Lemma~\ref{lemma:cubic-surface-codimension-one}.
But it follows from $\lambda<\frac{2}{3+3x}\leqslant\frac{2}{3}$ that the divisor
$$
-(K_{\mathbb{P}^2}+\lambda\overline{D}+\ell)\sim_{\mathbb{Q}} (3-3\lambda-1)\ell
$$
is ample. This contradicts Shokurov's connectedness principle \cite[Theorem~6.9]{Shokurov}.
\end{proof}

We may assume that $P\in E_1$.
Then $E_1$ is contained in the support of the divisor $D$,
because otherwise we would have $1-x=E_1\cdot D\geqslant\mathrm{mult}_{P}(D)>1$ by Lemma~\ref{lemma:Skoda}.

\begin{lemma}
\label{lemma:cubic-surface-Eckardt-point}
The point $P$ is an Eckardt point.
\end{lemma}

\begin{proof}
Let $T$ be the plane section of $S$ that is singular at $P$,
let $Z_1$, $Z_2$, $Z_3$, $Z_4$, $Z_5$, and $Z_6$ be general conics in
the pencils $|-K_{S}-E_1|$, $|-K_{S}-E_2|$, $|-K_{S}-E_3|$, $|-K_{S}-E_4|$, $|-K_{S}-E_5|$, and  $|-K_{S}-E_6|$, respectively.
Let
$$
\Upsilon=\Bigg(\frac{1-\lambda}{6}-\lambda x\Bigg)\sum_{i=1}^6E_i+\frac{1-\lambda}{6}\Big(Z_1+Z_2+Z_3+Z_4+Z_5+Z_6\Big).
$$
Then $\Upsilon$ is an effective $\mathbb{Q}$-divisor such that $\Upsilon+\lambda D\sim_{\mathbb{Q}}-K_{S}$,
and the pair $(S,\lambda D+\Upsilon)$ is not log canonical at $P$.
Then $P$ is not an ordinary double point of the curve $T$ by \cite[Corollary~1.24]{CheltsovParkWon}.
Thus, either $P$ is an Eckardt point or $T$ has a tacnodal singularity at the point $P$.

Suppose that $P$ is not an Eckardt point. Then $T=E_1+C$,
where $C$ is smooth conic such that $E_1$ is tangent to $C$ at the point $P$.
We have
$$
\lambda D\sim_{\mathbb{Q}}(\lambda+\lambda x)E_1+\lambda C+\lambda x\big(E_2+E_3+E_4+E_5+E_6\big),
$$
and $(S, (\lambda+\lambda x)E_1+\lambda C+\lambda x(E_2+E_3+E_4+E_5+E_6))$ is log canonical.
Therefore, by Remark~\ref{remark:convexity}, we may assume that the support of the divisor $D$ does not contain either $C$ or $E_2$ (or both).

Write $D=aE_1+bC+\Delta$, where $a$ and $b$ are non-negative rational numbers,
and $\Delta$ is an effective $\mathbb{Q}$-divisor,
whose support does not contain the curves $E_1$ and $C$.
We know that $a>0$.
If $b>0$, then $1-x-b\geqslant E_2\cdot\Delta\geqslant 0$, so that $b\leqslant 1-x$. Thus, we always have $b\leqslant 1-x$.

Let $m=\mathrm{mult}_P(\Delta)$. Then $1-x+a-2b=E_1\cdot\Delta\geqslant m$ and
$2+7x-2a=C\cdot\Delta\geqslant m$, which gives
$a\leqslant 1+\frac{7}{2}x$ and
$$
m+a+b\leqslant\frac{5+13x-3a-2b}{3}+a+b=\frac{5+13x}{3}+\frac{b}{3}\leqslant\frac{5+13x}{3}+\frac{1-x}{3}=2+4x<\frac{2}{\lambda}
$$
and $m\leqslant\frac{2-2x+2a-4b}{3}+\frac{2+7x-2a}{3}=\frac{4+5x-4b}{3}\leqslant \frac{4+5x}{3}<\frac{1}{\lambda}$.

Let $f\colon\widetilde{S}\to S$ be the blow up at $P$, and let $F$ be its exceptional curve.
Denote by $\widetilde{E}_1$, $\widetilde{C}$ and $\widetilde{\Delta}$ the proper transforms on $\widetilde{S}$ of the curve $E_1$, the curve $C$ and the divisor $\Delta$, respectively.
Then
$(\widetilde{S}, \lambda a\widetilde{E}_1+\lambda b\widetilde{C}+\lambda\widetilde{\Delta}+(\lambda(a+b+m)-1)F)$
is not log canonical at some point $Q\in F$.

If $Q\ne\widetilde{E}_1\cap\widetilde{C}$, then Lemma~\ref{lemma:adjunction} gives $m=F\cdot\widetilde{\Delta}>\frac{1}{\lambda}$, which is impossible, as $m<\frac{1}{\lambda}$. Thus $Q=\widetilde{E}_1\cap\widetilde{C}$. 
Let $\widetilde{m}=\mathrm{mult}_{Q}(\widetilde{\Delta})$.
Then
$1-x+a-2b-m=\widetilde{E}_1\cdot\widetilde{\Delta}\geqslant\widetilde{m} \text{ and } 2+7x-2a-m\geqslant\widetilde{C}\cdot\widetilde{\Delta}\geqslant\widetilde{m}$.
Since $b\leqslant 1-x$, the latter inequality gives
$$
m+\widetilde{m}+2a+2b\leqslant 2+7x+2b\leqslant 2+7x+2(1-x)\leqslant 4+5x<\frac{3}{\lambda}.
$$

Let $g\colon\widehat{S}\to\widetilde{S}$ be the blow up at $Q$, and let $G$ be its exceptional curve.
Denote by $\widehat{E}_1$, $\widehat{C}$, $\widehat{F}$ and $\widehat \Delta$ the proper transforms on $\widehat S$
of $\widetilde{E}_1$, $\widetilde{C}$, $F$ and $\widetilde{\Delta}$, respectively.
Then the singularities of the log pair
$(\widehat{S}, \lambda a\widehat{E}_1+\lambda a\widehat{C}+\lambda\widehat{\Delta}+(\lambda(a+b+m)-1)\widehat{F}+(\lambda(2a+2b+m+\widetilde{m})-2)G)$
are not log canonical at some point $O\in G$.
Note that $1\geqslant \lambda(2a+2b+m+\widetilde{m})-2>0$.

The curves $\widehat{E}_1$, $\widehat{C}$, and $\widehat{F}$ are disjoint.
If $O\not\in\widehat{E}_1\cup\widehat{C}\cup\widehat{F}$, then
$\widetilde{m}=G\cdot\widehat{\Delta}>\frac{1}{\lambda}$
by Lemma~\ref{lemma:adjunction}. But $\widetilde{m}\leqslant m<\frac{1}{\lambda}$,
so that $O\in\widehat{E}_1\cup\widehat{C}\cup\widehat{F}$.
If $O=\widehat{E}_1\cap G$, then Lemma~\ref{lemma:adjunction} implies
$$
1-x+a-2b-m-\widetilde{m}=\widehat{E_1}\cdot\widehat{\Delta}>\frac{3}{\lambda}-2a-2b-m-\widetilde{m},
$$
so that we have
$3\frac{2+7x}{2}+1-x\geqslant 3a+1-x>\frac{3}{\lambda}>\frac{3(3+3x)}{2}$,
which is impossible, because $x\leqslant\frac{1}{10}$.
If $O=\widehat{C}\cap G$, then Lemma~\ref{lemma:adjunction} implies
$$
2+7x-2a-m-\widetilde{m}\geqslant\widehat{C}\cdot\widehat{\Delta}>\frac{3}{\lambda}-2a-2b-m-\widetilde{m},
$$
so that
$4+5x=2+7x+2(1-x)\geqslant 2+7x+2b>\frac{3}{\lambda}>\frac{3(3+x)}{2}$,
which is impossible, because~$x\leqslant\frac{1}{10}$.
Finally, if we have $O=\widehat{F}\cap G$, then Lemma~\ref{lemma:adjunction} implies
$m-\widetilde{m}=\widehat{F}\cdot\widehat{\Delta}>\frac{3}{\lambda}-2a-2b-m-\widetilde{m}$,
so that $2m+2a+2b>\frac{3}{\lambda}$, which is impossible because $m+a+b\leqslant 2+4x$ and $x\leqslant\frac{1}{10}$.
\end{proof}

We see that $S$ contains two more lines (except the line $E_1$) that pass through the point $P$.
Denote these two lines by $L_1$ and $L_2$.
Recall that $E_1$ is contained in the support of $D$.
If $L_1$ is not contained in the support of $D$, then  Lemma~\ref{lemma:Skoda} gives
$$
1+6x\geqslant 1+xL_1\cdot\big(E_1+E_2+E_3+E_4+E_5+E_6\big)=L_1\cdot D\geqslant\mathrm{mult}_{P}(D)>\frac{1}{\lambda}>\frac{3+3x}{2}.
$$
However, $1+6x<\frac{3+3x}{2}$, because $x\leqslant\frac{1}{10}$.
Thus, we see that $L_1$ is also contained in the support of $D$.
Similarly, we see that $L_2$ is contained in the support of $D$ as well.

As usual, we write $D=aE_1+b_1L_1+b_2L_2+\Delta$, where $a$, $b_1$, $b_2$ are positive rational numbers,
and $\Delta$ is an effective $\mathbb{Q}$-divisor whose support does not contain $E_1$, $L_1$, or $L_2$.
We have
$$
\lambda D\sim_{\mathbb{Q}}(\lambda+\lambda x)E_1+\lambda L_1+\lambda L_2+\lambda x\big(E_2+E_3+E_4+E_5+E_6\big),
$$
and the singularities of the log pair $(S, (\lambda+\lambda x)E_1+\lambda L_1+\lambda L_2+\lambda x(E_2+E_3+E_4+E_5+E_6))$ are log canonical.
Hence, by Remark~\ref{remark:convexity}, we may assume that $\mathrm{Supp}(\Delta)$ does not contain at least one line among
$E_2$, $E_3$, $E_4$, $E_5$, $E_6$.

\begin{lemma}
\label{lemma:cubic-surface-b1-b2}
Either $b_1\leqslant 1-x$ or $b_2\leqslant 1-x$ (or both).
\end{lemma}

\begin{proof}
Without loss of generality, we may assume that $E_2\not\subset\mathrm{Supp}(D)$.
Then
$$
1-x-E_2\cdot\big(b_1L_1+b_2L_2\big)=E_2\cdot\Delta\geqslant 0.
$$
But
$E_2\cdot (L_1+L_2)=1$,
so that either $E_2\cdot L_1=1$ and $E_2\cdot L_2=0$,
or  $E_2\cdot L_1=0$ and $E_2\cdot L_2=1$.
In the former case, we get $b_1\leqslant 1-x$.
In the latter case, we get $b_2\leqslant 1-x$.
\end{proof}

We may assume that $h(L_1)$ is a line in $\mathbb{P}^2$, and $h(L_2)$ is a conic.
Then $h(L_1)$ is tangent to the conic $h(L_2)$  at the point $h(E_1)$.
We may also assume that $h(L_1)$ contains the point $h(E_6)$.
Then $L_1\cdot E_1=L_1\cdot E_6=1$,
$L_1\cdot E_2=L_1\cdot E_3=L_1\cdot E_4=L_1\cdot E_5=0$,
$L_2\cdot E_6=0$ and $L_2\cdot E_1=L_2\cdot E_2=L_2\cdot E_3=L_2\cdot E_4=L_2\cdot E_5=1$.

Let $m=\mathrm{mult}_P(\Delta)$. Then $E_1\cdot\Delta\geqslant m$, $L_1\cdot\Delta\geqslant m$ and $L_2\cdot\Delta\geqslant m$.
This gives
\begin{equation}
\label{equation:cubic-surfaces}
\left\{%
\aligned
&1-x+a-b_1-b_2\leqslant m,\\%
&1+2x-a+b_1-b_2\leqslant m,\\%
&1+5x-a-b_1+b_2\leqslant m.\\%
\endaligned\right.
\end{equation}

\begin{lemma}
\label{lemma:cubic-surface-inequalities}
One has $m+a\leqslant 1+\frac{7}{2}x$, $m+b_1\leqslant 1+2x$, $m+b_2\leqslant 1+\frac{x}{2}$ and $a+b_1+b_2+m\leqslant 3+3x$.
\end{lemma}

\begin{proof}
The first three inequalities directly follow from \eqref{equation:cubic-surfaces}.
To prove the fourth inequality, let us recall that $b_1\leqslant 1-x$ or $b_2\leqslant 1-x$ by Lemma~\ref{lemma:cubic-surface-b1-b2}.
Thus, if $b_1\leqslant 1-x$, then
$$
3+x\geqslant 3\geqslant 1+2x+2b_1\geqslant m+a+b_2+b_1
$$
by \eqref{equation:cubic-surfaces}.
Similarly, if $b_2\leqslant 1-x$, then \eqref{equation:cubic-surfaces} gives $3+3x\geqslant 1+5x+2b_2\geqslant m+a+b_2+b_1$.
\end{proof}

Let $f\colon\widetilde{S}\to S$ be a blow up of the point $P$. Denote by $F$ be the $f$-exceptional curve,
and denote by $\widetilde{E}_1$, $\widetilde{L}_1$ and $\widetilde{L}_2$ the proper transforms on $\widetilde{S}$ of the lines $E_1$, $L_1$ and $L_2$, respectively.
Similarly, denote by $\widetilde{\Delta}$ the proper transform on $\widetilde{S}$ of the $\mathbb{Q}$-divisor $\Delta$.
Then the singularities of the log pair
$(\widetilde{S}, \lambda a\widetilde{E}_1+\lambda b_1\widetilde{L}_1+\lambda b_1\widetilde{L}_1+\lambda\widetilde{\Delta}+(\lambda(a+b_1+b_2+m)-1)F)$
are not log canonical at some point $Q\in F$. By Lemma~\ref{lemma:cubic-surface-inequalities},
we have $m+a+b_1+b_2\leqslant 3+3x<\frac{2}{\lambda}$, since $\lambda<\frac{2}{3+3x}$.

Lemma~\ref{lemma:adjunction} gives
$$
m=F\cdot\widetilde{\Delta}>
\left\{%
\aligned
&\frac{3+3x}{2}-a\ \text{if $Q=F\cap\widetilde{E}_1$},\\%
&\frac{3+3x}{2}-b_1\ \text{if $Q=F\cap\widetilde{L}_1$},\\%
&\frac{3+3x}{2}-b_2\ \text{if $Q=F\cap\widetilde{L}_2$},\\%
&\frac{3+3x}{2}\ \text{if $Q\not\in\widetilde{E}_1\cup\widetilde{L}_1\cup \widetilde{L}_2$},\\%
\endaligned\right.%
$$
because $\lambda<\frac{2}{3+3x}$.
But $m+a\leqslant 1+\frac{7}{2}x$, $m+b_1\leqslant 1+2x$ and $m+b_2\leqslant 1+\frac{x}{2}$ by Lemma~\ref{lemma:cubic-surface-inequalities}.
This immediately leads to a contradiction, because $0\leqslant x\leqslant\frac{1}{10}$.

The obtained contradiction completes the proof of Theorem~\ref{theorem:cubic-surface-1}.

\section{Del Pezzo surfaces of degree four and higher}
\label{section:deg-4}

Let $S$ be a smooth del Pezzo surface such that $K_{S}^2\geqslant 4$. Let $d=K_{S}^2$.
Let $L$ be an ample $\mathbb{Q}$-divisor on the surface $S$.
Let $\nu(L)=\frac{-K_S\cdot L}{L^2}$.
The goal of this section is to prove Theorem~\ref{theorem:big-degree}.

If $S$ is one of the surfaces $\mathbb{P}^2$, $\mathbb P^1\times\mathbb P^1$ or $\mathbb{F}_1$, then one easily sees that $\alpha(S,L)<\frac{2}{3}\nu(L)$.
Thus, to prove Theorem~\ref{theorem:big-degree}, we may assume $d\leqslant 7$.
In this case the Mori cone $\overline{\mathbb{NE}(S)}$ is a polyhedral cone that is generated by all $(-1)$-curves in $S$. Let
$$
\mu_L=\mathrm{inf}\Big\{\lambda\in\mathbb{Q}_{\geqslant 0}\ \Big|\ \text{the $\mathbb{Q}$-divisor}\ K_{S}+\lambda L\in\overline{\mathbb{NE}(S)}\Big\}.
$$
Replacing $L$ by $\mu_LL$, we may assume that $\mu_L=1$.
Denote by $\Delta_{L}$ the smallest extremal face of the Mori cone $\overline{\mathbb{NE}(S)}$ that contains the divisor $K_{S}+L$,
and denote its dimension by $r_L$.

If $\Delta_L=0$ and $d=4$, then $\alpha(S,L)=\frac{2}{3}\nu(L)$ by \cite[Theorem~1.7]{Ch08}.
Similarly, we see that if $\Delta_L=0$ and $d\geqslant 5$, then $\alpha(S,L)\leqslant\frac{1}{2}\nu(L)$.
Thus, we may assume that $\Delta_L\ne 0$, so that $r_L>0$.

Let $\phi_L\colon S\to Z$ be the contraction of the face $\Delta_{L}$.
Then either $\phi_L$ is a birational morphism, or $\phi_L$ is a conic bundle with $Z=\mathbb{P}^1$.

\begin{lemma}
\label{lemma:degree-four-birational}
Suppose that $Z=\mathbb{P}^2$.
Then $\alpha(S,L)<\frac{2}{3}\nu(L)$.
\end{lemma}

\begin{proof}
The face $\Delta_L$ is generated by $r_L=9-d$ disjoint $(-1)$-curves $E_1,\ldots,E_{r_L}$.
Then
$$
L\sim_{\mathbb{Q}}-K_{S}+\sum_{i=1}^{r_L}a_iE_i
$$
for some positive rational numbers $a_1,\ldots,a_{r_L}$.
We may assume that $a_1\geqslant\cdots\geqslant a_{r_L}$.
Then we must have $a_{1}<1$, because  $L\cdot E_{1}>0$.
We also have
$$
\nu(L)=\frac{d+\sum_{i=1}^{r_L}a_i}{d+2\sum_{i=1}^{r_L}a_i-\sum_{i=1}^{r_L}a_i^2}.
$$

For every $i\leqslant r_L$, denote by $L_{1i}$ the proper transform on $S$ of the line in $\mathbb{P}^2$ that passes through the points $\phi_L(E_1)$ and $\phi_L(E_i)$,
e.g., $\phi_L(L_{12})$ is the line that contains $\phi_L(E_1)$ and $\phi_L(E_2)$.

If $d=7$, then $L\sim_{\mathbb{Q}} 3L_{12}+(2+a_1)E_1+(2+a_2)E_2$, so that
$$
\alpha(S,L)\leqslant\frac{1}{3}<\frac{14+2a_1+2a_2}{21+6a_1+6a_2-3a_1^2-3a_2^2}=\frac{2}{3}\nu(L).
$$

Similarly, if $d=6$, then
$L\sim_{\mathbb{Q}} 2L_{12}+L_{13}+(2+a_1)E_1+(1+a_2)E_2+a_3E_3$,
which implies that
$$
\alpha(S,L)\leqslant\frac{1}{2+a_1}<\frac{12+2a_1+2a_2+2a_3}{18+6a_1+6a_2+6a_3-3a_1^2-3a_2^2-3a_3^2}=\frac{2}{3}\nu(L),
$$
because otherwise we would have
$$
0\geqslant 6 - 2\sum_{i=1}^3a_i + 12a_1 + 2a_1\sum_{i=1}^3a_i + 3\sum_{i=1}^3 a_1^2\geqslant 6+2\sum_{i=1}^3a_i+2a_1\sum_{i=1}^3a_i+3\sum_{i=1}^3a_1^2,
$$
which is absurd. If $d=5$, then $L\sim_{\mathbb{Q}} L_{12}+L_{13}+L_{14}+(2+a_1)E_1+a_2E_2+a_3E_3+a_4E_4$, which implies that
$$
\alpha(S,L)\leqslant\frac{1}{2+a_1}<\frac{10+2a_1+2a_2+2a_3+2a_4}{15+6a_1+6a_2+6a_3+6a_4-3a_1^2-3a_2^2-3a_3^2-3a_4^2}=\frac{2}{3}\nu(L).
$$
Indeed, if this inequality does not hold, then
$$
2\sum_{i=1}^4a_i\geqslant 5+3\sum_{i=1}^4 a_i^2+10a_1+2a_1\sum_{i=1}^4a_i>5a_1+3a_1^2+10a_1+2a_1^2,
$$
so that $8a_1\geqslant 2\sum_{i=1}^4a_i>15a_1+5a_1^2$, which is absurd.

We may assume that $d=4$.
Let $Z$ the proper transform on the surface $S$ of the conic in $\mathbb{P}^2$ that passes
through the points $\phi_L(E_1)$, $\phi_L(E_2)$, $\phi_L(E_3)$, $\phi_L(E_4)$, and $\phi_L(E_5)$. Then
$$
L\sim_{\mathbb{Q}}\frac{3+2a_1}{2}E_1+\frac{1}{2}Z+\frac{1}{2}L_{12}+\frac{1}{2}L_{13}+\frac{1}{2}L_{14}+\frac{1}{2}L_{15}+a_2E_2+a_3E_3+a_4E_4+a_5E_5,
$$
which implies that $\alpha(S,L)\leqslant\frac{2}{3+2a_1}$.
Let $N$ be the largest number among
$a_2$, $a_2+a_3$, $a_2+a_4$, $a_2+a_5$, $a_3+a_4$, $a_3+a_5$, $a_4+a_5$, $a_2+a_3+a_4$,
$a_2+a_3+a_5$, $a_2+a_4+a_5$, $a_3+a_4+a_5$, and $a_2+a_3+a_4+a_5$ that does not exceed $1$.
We claim that $\alpha(S,L)\leqslant\frac{2}{3+2a_1+N}$.
Indeed, we have $E_2\sim_{\mathbb{Q}}\frac{1}{2}E_1-\frac{1}{2}Z-\frac{1}{2}L_{12}+\frac{1}{2}L_{13}+\frac{1}{2}L_{14}+\frac{1}{2}L_{15}$, so that we have
$$
L\sim_{\mathbb{Q}}\frac{3+2a_1+a_2}{2}E_1+\frac{1-a_2}{2}Z+\frac{1-a_2}{2}L_{12}+\frac{1+a_2}{2}L_{13}+\frac{1+a_2}{2}L_{14}+\frac{1+a_2}{2}L_{15}+\sum_{i=3}^5a_iE_i.
$$
In particular, if $N=a_2$, then $\alpha(S,L)\leqslant\frac{2}{3+2a_1+N}$ as claimed.
We also have
$$
\left\{%
\aligned
&E_3\sim_{\mathbb{Q}}\frac{1}{2}E_1-\frac{1}{2}Z+\frac{1}{2}L_{12}-\frac{1}{2}L_{13}+\frac{1}{2}L_{14}+\frac{1}{2}L_{15},\\
&E_4\sim_{\mathbb{Q}}\frac{1}{2}E_1-\frac{1}{2}Z+\frac{1}{2}L_{12}+\frac{1}{2}L_{13}-\frac{1}{2}L_{14}+\frac{1}{2}L_{15},\\
&E_5\sim_{\mathbb{Q}}\frac{1}{2}E_1-\frac{1}{2}Z+\frac{1}{2}L_{12}+\frac{1}{2}L_{13}+\frac{1}{2}L_{14}-\frac{1}{2}L_{15}.
\endaligned\right.
$$
For any other value of $N$, Table \ref{table:Qdivisors-deg-4} provides an effective $\mathbb{Q}$-divisor
$D\sim_{\mathbb{Q}}L$ that shows the inequality $\alpha(S,L)\leqslant\frac{2}{3+2a_1+N}$ as claimed.

\begin{table}[!ht]%
\bgroup
\def\arraystretch{2}
\begin{tabular}{|c|c|c|}
\hline
$N$ & Effective $\mathbb Q$-divisor $D\sim_{\mathbb Q}L$ \\
\hline\hline
$a_2+a_3$ & $
\frac{3+2a_1+N}{2}E_1+\frac{1-N}{2}Z+\frac{1-a_2+a_3}{2}L_{12}+\frac{1+a_2-a_3}{2}L_{13}+{\displaystyle{\sum_{i=4,5}}}\left(\frac{1+a_2+a_3}{2}L_{1i}+a_iE_i\right)$\\
\hline
$a_2+a_4$ & $\frac{3+2a_1+N}{2}E_1+\frac{1-N}{2}Z+\frac{1-a_2+a_4}{2}L_{12}
+\frac{1+a_2-a_4}{2}L_{14}+{\displaystyle{\sum_{i=3,5}}}\left(\frac{1+a_2+a_4}{2}L_{1i}+a_iE_i\right)$\\
\hline
$a_2+a_5$ & $\frac{3+2a_1+N}{2}E_1+\frac{1-N}{2}Z+\frac{1-a_2+a_5}{2}L_{12}+\frac{1+a_2-a_5}{2}L_{15}+{\displaystyle{\sum_{i=3,4}}}\left(\frac{1+a_2+a_5}{2}L_{1i}+a_iE_i\right)$\\
\hline
$a_3+a_4$ & $\frac{3+2a_1+N}{2}E_1+\frac{1-N}{2}Z+\frac{1-a_3+a_4}{2}L_{13}+\frac{1+a_3-a_4}{2}L_{14}+{\displaystyle{\sum_{i=2,5}}}\left(\frac{1+a_3+a_4}{2}L_{1i}+a_iE_i\right)
$\\
\hline
$a_3+a_5$ & $\frac{3+2a_1+N}{2}E_1+\frac{1-N}{2}Z+\frac{1-a_3+a_5}{2}L_{13}+\frac{1+a_3-a_5}{2}L_{15}+{\displaystyle{\sum_{i=2,4}}}\left(\frac{1+a_3+a_5}{2}L_{1i}+a_iE_i\right)$\\
\hline
$a_4+a_5$ & $\frac{3+2a_1+N}{2}E_1+\frac{1-N}{2}Z+\frac{1-a_4+a_5}{2}L_{14}+\frac{1+a_4-a_5}{2}L_{15}+{\displaystyle{\sum_{i=2,3}}}\left(\frac{1+a_4+a_5}{2}L_{1i}+a_iE_i\right)$\\
\hline
$a_2+a_3+a_4$ & $\frac{3+2a_1+N}{2}E_1+\frac{1-N}{2}Z+\sum_{i=2,3,4}\left(\frac{1+\sum_{j=2,3,4}(\delta_{ij}a_j)}{2}L_{1i}\right)+\frac{1+N}{2}L_{15}+a_5E_5$\\
\hline
$a_2+a_4+a_5$ & $\frac{3+2a_1+N}{2}E_1+\frac{1-N}{2}Z+\sum_{i=2,4,5}\left(\frac{1+\sum_{j=2,4,5}(\delta_{ij}a_j)}{2}L_{1i}\right)+\frac{1+N}{2}L_{13}+a_3E_3$\\
\hline
$a_3+a_4+a_5$ & $\frac{3+2a_1+N}{2}E_1+\frac{1-N}{2}Z+\sum_{i=3,4,5}\left(\frac{1+\sum_{j=3,4,5}(\delta_{ij}a_j)}{2}L_{1i}\right)+\frac{1+N}{2}L_{12}+a_2E_2$\\
\hline
${\displaystyle{\sum_{i=2}^5}}a_i$ & $\frac{3+2a_1+N}{2}E_1+\frac{1-N}{2}Z+\sum_{i=2,3,4,5}\left(\frac{1+\sum_{j=2,3,4,5}(\delta_{ij}a_j)}{2}L_{1i}\right)$\\
\hline
\end{tabular}
\caption{Effective $\mathbb Q$-divisors, where $\delta_{ij}=1$ if $i\neq j$ and  $\delta_{ij}=-1$ if $i=j$.}
\label{table:Qdivisors-deg-4}
\egroup
\end{table}
Now Proposition~\ref{proposition:Simons-center} gives $\alpha(S,L)<\frac{2}{3}\nu(L)$.
\end{proof}

By Lemma~\ref{lemma:degree-four-birational}, we may assume that $Z\ne\mathbb{P}^2$ in order to complete the proof of Theorem~\ref{theorem:big-degree}.
Then the surface $S$ contains  $8-d\geqslant r_L$ disjoint $(-1)$-curves $E_1,\ldots,E_{8-d}$ and a smooth rational curve $C$ such that $C^2=0$,
the curve $C$ is disjoint from $E_1,\ldots,E_{8-d}$, and
$$
L\sim_{\mathbb{Q}}-K_{S}+\delta C+\sum_{i=1}^{8-d}a_iE_i
$$
for some non-negative rational numbers $\delta,a_1,\ldots,a_{8-d}$.
We may assume that $a_1\geqslant\cdots\geqslant a_{8-d}$.
Then $a_1<1$, since $L\cdot E_1>0$ and $C\cdot E_1=0$.
If $a_i>0$, then $E_i\in\Delta_L$.
Similarly, the morphism $\phi_L$ is a conic bundle if and only if $\delta>0$.
We have
$$
\nu(L)=\frac{d+2\delta+\sum_{i=1}^{8-d}a_i}{d+4\delta+2\sum_{i=1}^{8-d}a_i-\sum_{i=1}^{8-d}a_i^2}.
$$

Let $h\colon S\to\overline{S}$ be the contraction of the curves $E_1,\ldots,E_{8-d}$.
Then $\overline{S}=\mathbb{F}_1$ or  $\overline{S}=\mathbb{P}^1\times\mathbb{P}^1$.
The linear system $|C|$ is a free pencil.
It gives a conic bundle $\iota\colon S\to\mathbb{P}^1$ such that there exists a commutative diagram
$$
\xymatrix{
S\ar@{->}[rd]_\iota\ar@{->}[rr]^h&&\overline{S}\ar@{->}[ld]^\upsilon\\
&\mathbb{P}^1&}
$$
where $\upsilon$ is a natural projection. Then the curves $E_1,\ldots,E_{8-d}$ lie in the reducible fibers of $\iota$.

\begin{lemma}
\label{lemma:degree-four-conic-bundle-F-1}
Suppose that $\overline{S}=\mathbb{F}_1$.
Then $\alpha(S,L)<\frac{2}{3}\nu(L)$.
\end{lemma}

\begin{proof}
Let $\overline{E}_{9-d}$ be the $(-1)$-curve on $\overline{S}$, and let $E_{9-d}$ be its proper transform on the surface $S$.
Let $f\colon\overline{S}\to\mathbb{P}^2$ be the contraction of the curve $\overline{E}_{9-d}$. Let $\pi=f\circ g$.
Denote by $L_{1i}$ the proper transform on $S$ of the line in $\mathbb{P}^2$
that passes through $\pi(E_1)$ and $\pi(E_i)$ (for every $i\leqslant 9-d$).

If $d=7$, then $L\sim_{\mathbb{Q}} (3+\delta)L_{12}+(2+\delta+a_1)E_1+2E_2$ and
$$
\alpha(S,L)\leqslant\frac{1}{3+\delta}<\frac{14+4\delta+2a_1}{21+12\delta+6a_1-3a_1^2}=\frac{2}{3}\nu(L),
$$
since otherwise we would have $0\geqslant 21+14\delta+4\delta^2+2\delta a_1+3a_1^2$.

If $d=6$, then
$L\sim_{\mathbb{Q}} 2L_{12}+(1+\delta)L_{13}+(2+\delta+a_1)E_1+(1+a_2)E_2$,
which gives
$$
\alpha(S,L)\leqslant\frac{1}{2+\delta+a_1}<\frac{12+4\delta+2a_1+2a_2}{18+12\delta+6a_1+6a_2-3a_1^2-3a_2^2}=\frac{2}{3}\nu(L),
$$
because otherwise we would have
$$
0\geqslant6-2(a_1+a_2)+3(a_1^2+a_2^2)+8\delta+12a_1+(\delta+a_1)(4\delta+2a_1+2a_2),
$$
which gives a contradiction, since $a_1\geqslant a_2$.

If $d=5$, then $L\sim_{\mathbb{Q}}L_{12}+L_{13}+(1+\delta)L_{14}+(2+\delta+a_1)E_1+a_2E_2+a_3E_3$,
which gives
$$
\alpha(S,L)\leqslant\frac{1}{2+\delta+a_1}<\frac{10+4\delta+2a_1+2a_2+2a_3}{15+12\delta+6a_1+6a_2+6a_3-3a_1^2-3a_2^2-3a_3^2}=\frac{2}{3}\nu(L).
$$
Indeed, if this inequality does not hold, then we have
$$
0\geqslant 5+6\delta-2\sum_{i=1}^3a_i+3\sum_{i=1}^3a_i^2+10a_1+4\delta(\delta+a_1)+2(\delta+a_1)\sum_{i=1}^3a_i,
$$
which is impossible since  $-2\sum_{i=1}^3a_i+10a_1\geqslant 4a_1$ as $a_1\geqslant a_2\geqslant a_3$.

Thus, we may assume that $d=4$.
Denote by $Z$ be the proper transform on the surface $S$ of the unique conic in $\mathbb{P}^2$ that passes
through the points $\pi(E_1)$, $\pi(E_2)$, $\pi(E_3)$, $\pi(E_4)$, $\pi(E_5)$.
Then
$$
L\sim_{\mathbb{Q}}\delta C+\frac{3+2a_1}{2}E_1+\frac{1}{2}\big(Z+L_{12}+L_{13}+L_{14}+L_{15}\big)+a_2E_2+a_3E_3+a_4E_4,
$$
Moreover, there exists a $(-1)$-curve $E_1^\prime$ such that $E_1^\prime+E_1\sim C$. In fact $E_1'=L_{15}$. Hence
$$
L\sim_{\mathbb{Q}}\frac{3+2a_1+2\delta}{2}E_1+\delta E_1^\prime+\frac{1}{2}\big(Z+L_{12}+L_{13}+L_{14}+L_{15}\big)+a_2E_2+a_3E_3+a_4E_4.
$$
Arguing as in the proof of Lemma~\ref{lemma:degree-four-birational}, we see that
$$
\alpha(S,L)\leqslant\frac{2}{3+2a_1+2\delta+N},
$$
where $N$ is the largest number among $a_2$, $a_2+a_3$, $a_2+a_4$, $a_3+a_4$, $a_2+a_3+a_4$, that does not exceed $1$.
Then $\alpha(S,L)<\frac{2}{3}\nu(L)$ by Proposition~\ref{proposition:Simons-center}.
\end{proof}

The final step of the proof of Theorem~\ref{theorem:big-degree} is the following Lemma.

\begin{lemma}
\label{lemma:degree-four-conic-bundle-F-0}
Suppose that $\overline{S}=\mathbb{P}^1\times\mathbb{P}^1$.
Then $\alpha(S,L)<\frac{2}{3}\nu(L)$.
\end{lemma}

\begin{proof}
We may assume that the fibers of the natural projection $\upsilon\colon\overline{S}\to\mathbb{P}^1$ are curves of bi-degree $(1,0)$.
For every $i\leqslant 8-d$, we denote by $\overline{F}_i$ (respectively, by $\overline{F}_i^\prime$)
the curve of bi-degree $(1,0)$ (respectively, bi-degree $(1,0)$) in $\mathbb{P}^1\times\mathbb{P}^1$
that passes through the point $h(E_i)$, and we denote by $F_i$ (respectively, by $F_i^\prime$) its proper transform on $S$.

If $d=7$, then
$L\sim_{\mathbb{Q}} (3+a_1+\delta)E_1+(2+\delta)F_1+2F_1^\prime$, which gives
$$
\alpha(S,L)\leqslant\frac{1}{3+a_1+\delta}<\frac{14+4\delta+2a_1}{21+12\delta+6a_1-3a_1^2}=\frac{2}{3}\nu(L).
$$
Similarly, if $d=6$, then
$L\sim_{\mathbb{Q}}(\frac{3}{2}+\delta)F_1+\frac{3}{2}F_1^\prime+\frac{1}{2}(F_2+F_2^\prime)+(2+\delta+a_1)E_1+a_2E_2$,
which gives
$$
\alpha(S,L)\leqslant\frac{1}{2+\delta+a_1}<\frac{12+4\delta+2a_1+2a_2}{18+12\delta+6a_1+6a_2-3a_1^2-3a_2^2}=\frac{2}{3}\nu(L).
$$

For $d\leqslant 5$, denote by $Z_{ij}$ for $2\leqslant i<j\leqslant 4$ the proper transform on $S$ of the unique irreducible curve of bi-degree $(1,1)$ on $\overline{S}$
that passes through the points $h(E_1)$, $h(E_i)$ and $h(E_j)$.
If $d=5$, then
$$
L\sim_{\mathbb{Q}}(1+\delta)F_1+F_1^\prime+Z_{23}+(2+\delta+a_1)E_1+a_2E_2+a_3E_3,
$$
so that, arguing as in the proof of Lemma~\ref{lemma:degree-four-conic-bundle-F-1}, we see that
$$
\alpha(S,L)\leqslant\frac{1}{2+\delta+a_1}<\frac{10+4\delta+2a_1+2a_2+2a_3}{15+12\delta+6a_1+6a_2+6a_3-3a_1^2-3a_2^2-3a_3^2}=\frac{2}{3}\nu(L).
$$
Thus, we may assume that $d=4$. Then $-K_S\sim_{\mathbb{Q}}\frac{3}{2}E_1+\frac{1}{2}\big(F_1+F_1^\prime+Z_{23}+Z_{24}+Z_{34}\big)$. Hence
$$
L\sim_{\mathbb{Q}}\frac{3+2a_1+2\delta}{2}E_1+\frac{1+2\delta}{2}F_1+\frac{1}{2}F_1^\prime+\frac{1}{2}Z_{23}+\frac{1}{2}Z_{24}+\frac{1}{2}Z_{34}+a_2E_2+a_3E_3+a_4E_4,
$$
which implies, in particular, that $\alpha(S,L)\leqslant\frac{2}{3+2a_1+2\delta}$.
Moreover, we have
$$
\left\{%
\aligned
&E_2\sim_{\mathbb{Q}}\frac{1}{2}E_1+\frac{1}{2}F_1+\frac{1}{2}F_1^\prime-\frac{1}{2}Z_{23}-\frac{1}{2}Z_{24}+\frac{1}{2}Z_{34},\\
&E_3\sim_{\mathbb{Q}}\frac{1}{2}E_1+\frac{1}{2}F_1+\frac{1}{2}F_1^\prime-\frac{1}{2}Z_{23}+\frac{1}{2}Z_{24}-\frac{1}{2}Z_{34},\\
&E_4\sim_{\mathbb{Q}}\frac{1}{2}E_1+\frac{1}{2}F_1+\frac{1}{2}F_1^\prime+\frac{1}{2}Z_{23}-\frac{1}{2}Z_{24}-\frac{1}{2}Z_{34}.\\
\endaligned\right.
$$
Therefore, we have
\begin{multline*}
L\sim_{\mathbb{Q}}\frac{3+2a_1+2\delta+a_2}{2}E_1+\frac{1+2\delta+a_2}{2}F_1+\frac{1+a_2}{2}F_1^\prime+\\
+\frac{1-a_2}{2}Z_{23}+\frac{1-a_2}{2}Z_{24}+\frac{1+a_2}{2}Z_{34}+a_3E_3+a_4E_4,
\end{multline*}
so that $\alpha(S,L)\leqslant\frac{2}{3+2a_1+2\delta+a_2}$.
Moreover, if $a_2+a_3\leqslant 1$, then $\alpha(S,L)\leqslant\frac{2}{3+2a_1+2\delta+a_2+a_3}$, since
\begin{multline*}
L\sim_{\mathbb{Q}}\frac{3+2a_1+2\delta+a_2+a_3}{2}E_1+\frac{1+2\delta+a_2+a_3}{2}F_1+\frac{1+a_2+a_3}{2}F_1^\prime+\\
+\frac{1-a_2-a_3}{2}Z_{23}+\frac{1-a_2+a_3}{2}Z_{24}+\frac{1+a_2-a_3}{2}Z_{34}+a_4E_4.
\end{multline*}
Similarly, if $a_2+a_4\leqslant 1$, then $\alpha(S,L)\leqslant\frac{2}{3+2a_1+2\delta+a_2+a_4}$, since
\begin{multline*}
L\sim_{\mathbb{Q}}\frac{3+2a_1+2\delta+a_2+a_4}{2}E_1+\frac{1+2\delta+a_2+a_4}{2}F_1+\frac{1+a_2+a_4}{2}F_1^\prime+\\
+\frac{1-a_2+a_4}{2}Z_{23}+\frac{1-a_2-a_4}{2}Z_{24}+\frac{1+a_2-a_4}{2}Z_{34}+a_3E_3.
\end{multline*}
And if $a_3+a_4\leqslant 1$, then $\alpha(S,L)\leqslant\frac{2}{3+2a_1+2\delta+a_3+a_4}$, since
\begin{multline*}
L\sim_{\mathbb{Q}}\frac{3+2a_1+2\delta+a_3+a_4}{2}E_1+\frac{1+2\delta+a_3+a_4}{2}F_1+\frac{1+a_3+a_4}{2}F_1^\prime+\\
+\frac{1-a_3+a_4}{2}Z_{23}+\frac{1+a_3-a_4}{2}Z_{24}+\frac{1-a_3-a_4}{2}Z_{34}+a_2E_2.
\end{multline*}
Finally, if $a_2+a_3\leqslant 1+a_4$,
then $\alpha(S,L)\leqslant\frac{2}{3+2a_1+2\delta+a_2+a_3+a_4}$, since
\begin{multline*}
L\sim_{\mathbb{Q}}\frac{3+2a_1+2\delta+a_2+a_3+a_4}{2}E_1+\frac{1+2\delta+a_2+a_3+a_4}{2}F_1+\frac{1+a_2+a_3+a_4}{2}F_1^\prime+\\
+\frac{1-a_2-a_3+a_4}{2}Z_{23}+\frac{1-a_2+a_3-a_4}{2}Z_{24}+\frac{1+a_2-a_3-a_4}{2}Z_{34}.
\end{multline*}
Thus, we proved that
$$
\alpha(S,L)\leqslant\left\{%
\aligned
&\frac{2}{3+2a_1+2\delta+a_2+a_3+a_4}\ \text{if $a_2+a_3\leqslant 1+a_4$},\\%
&\frac{2}{3+2a_1+2\delta+a_2+a_4}\ \text{if $a_2+a_4\leqslant 1$ and $a_2+a_3>1+a_4$},\\%
&\frac{2}{3+2a_1+2\delta+a_3+a_4}\ \text{if $a_3+a_4\leqslant 1$, $a_2+a_4>1$ and $a_2+a_3>1+a_4$},\\%
&\frac{2}{3+2a_1+2\delta+a_2}\ \text{if $a_3+a_4>1$, $a_2+a_4>1$ and $a_2+a_3>1+a_4$}.\\%
\endaligned\right.%
$$
By Proposition~\ref{proposition:Simons-center}, $\alpha(S,L)<\frac{8+4\delta+2a_1+2a_2+2a_3+2a_4}{12+12\delta+6a_1+6a_2+6a_3+6a_4-3a_1^2-3a_2^2-3a_3^2-3a_4^2}=\frac{2}{3}\nu(L)$.
\end{proof}

\appendix
\section{Inequalities}
\label{section:inequalities}

Let $a_1$, $a_2$, $a_3$, $a_4$, $a_5$ be non-negative rational numbers such that $1\geqslant a_1\geqslant a_2\geqslant a_3\geqslant a_4\geqslant a_5$,
let $\delta $ be a non-negative rational number,
and let $N$ be the largest number among $a_2$, $a_2+a_3$, $a_2+a_4$, $a_2+a_5$, $a_3+a_4$, $a_3+a_5$, $a_4+a_5$, $a_2+a_3+a_4$, $a_2+a_3+a_5$,
$a_2+a_4+a_5$, $a_3+a_4+a_5$, $a_2+a_3+a_4+a_5$ that does not exceed $1$.
Let
$$
\alpha=\left\{%
\aligned
&\frac{2}{3+2a_1+2\delta+a_2+a_3+a_4}\ \text{if $a_2+a_3\leqslant 1+a_4$},\\%
&\frac{2}{3+2a_1+2\delta+a_2+a_4}\ \text{if $a_2+a_4\leqslant 1$ and $a_2+a_3>1+a_4$},\\%
&\frac{2}{3+2a_1+2\delta+a_3+a_4}\ \text{if $a_3+a_4\leqslant 1$, $a_2+a_4>1$ and $a_2+a_3>1+a_4$},\\%
&\frac{2}{3+2a_1+2\delta+a_2}\ \text{if $a_3+a_4>1$ and $a_2+a_3>1+a_4$}.\\%
\endaligned\right.%
$$

\begin{proposition}
\label{proposition:Simons-center}
One has
$$
\frac{2}{3+2a_1+2\delta+N}\leqslant\frac{2}{3}\cdot\frac{4+2\delta+a_1+a_2+a_3+a_4+a_5}{4+4\delta+2(a_1+a_2+a_3+a_4+a_5)-a_1^2-a_2^2-a_3^2-a_4^2-a_5^2}
$$
and
$$
\alpha\leqslant\frac{8+4\delta+2a_1+2a_2+2a_3+2a_4}{12+12\delta+6a_1+6a_2+6a_3+6a_4-3a_1^2-3a_2^2-3a_3^2-3a_4^2}.
$$
Moreover, both inequalities are strict unless $a_1=\delta=0$.
\end{proposition}

\begin{proof}
Consider the first inequality. If $a_1=\delta=0$, then $N=0$ and both sides of the inequality equal $\frac{2}{3}$. Suppose that either $a_1>0$ or $\delta>0$ (or both). We have to prove that
$$
\frac{2}{3+2a_1+2\delta+N}<\frac{2}{3}\times\frac{4+2\delta+a_1+a_2+a_3+a_4+a_5}{4+4\delta+2(a_1+a_2+a_3+a_4+a_5)-a_1^2-a_2^2-a_3^2-a_4^2-a_5^2}
$$
Suppose that this inequality does not hold. Then
$$
3\sum_{i=1}^5a_i-3\sum_{i=1}^5a_i^2-2\delta\geqslant (2a_1+N+2\delta)\sum_{i=1}^5a_i+8a_1+4N+4a_1\delta+2N\delta+4\delta^2.
$$
Since either $a_1>0$ or $\delta>0$, this inequality implies that
\begin{equation}
\label{equation:Simons-center}
3\sum_{i=1}^5a_i>(2a_1+N)\sum_{i=1}^5a_i+8a_1+4N.
\end{equation}
If $N=a_2+a_3+a_4+a_5$, then \eqref{equation:Simons-center} gives
$$
3\sum_{i=1}^5a_i>(2a_1+N)\sum_{i=1}^5a_i+8a_1+4N\geqslant 8a_1+4N\geqslant 4\sum_{i=1}^5a_i,
$$
which is absurd. Thus, $N\ne a_2+a_3+a_4+a_5$, so that $a_2+a_3+a_4+a_5>1$. Then \eqref{equation:Simons-center} gives
$$
3\sum_{i=1}^5a_i>(2a_1+N)\sum_{i=1}^5a_i+8a_1+4N>(2a_1+N)(a_1+1)+8a_1+4N,
$$
which implies, in particular, that
\begin{equation}
\label{equation:Simons-center-weak}
3(a_2+a_3+a_4+a_5)>7a_1+5N,
\end{equation}
which implies that $N$ is none of the numbers $a_2+a_3+a_4$, $a_2+a_3+a_5$, $a_2+a_4+a_5$, $a_3+a_4+a_5$.
If $N$ is one of the numbers $a_2+a_3$, $a_2+a_4$, $a_2+a_5$, $a_3+a_4$, $a_3+a_5$, $a_4+a_5$, then \eqref{equation:Simons-center-weak} gives
$$
3(a_2+a_3+a_4+a_5)>7a_1+5N\geqslant 7a_1+5(a_4+a_5),
$$
so that $3a_2+3a_3>7a_1$, which is absurd. We see that $N=a_2$. Then \eqref{equation:Simons-center-weak} gives
$$
3(a_3+a_4+a_5)>7a_1+2a_2,
$$
which is impossible, because $a_1\geqslant a_2\geqslant a_3\geqslant a_4\geqslant a_5\geqslant 0$.

Consider the second inequality. If $a_1=\delta=0$, then both sides of the inequality equal $\frac{2}{3}$.
Suppose that $a_1>0$ or $\delta>0$. Take $N\geqslant 0$ such that $\alpha=\frac{2}{3+2a_1+2\delta +N}$.
We have to prove that
$$
\alpha<\frac{8+4\delta+2a_1+2a_2+2a_3+2a_4}{12+12\delta+6a_1+6a_2+6a_3+6a_4-3a_1^2-3a_2^2-3a_3^2-3a_4^2}.
$$
Suppose the latter is false. Then
$$
3\sum_{i=1}^4a_i-3\sum_{i=1}^4a_i^2-2\delta\geqslant (2a_1+N+2\delta)\sum_{i=1}^4a_i+8a_1+4N+4a_1\delta+2N\delta+4\delta^2.
$$
Since either $a_1>0$ or $\delta>0$, this inequality implies that
$$
3\sum_{i=1}^4a_i>(2a_1+N)\sum_{i=1}^4a_i+8a_1+4N.
$$
If $N=a_2+a_3+a_4$ or $N=a_2+a_4$, then we get a contradiction
$$
3\sum_{i=1}^4 a_i > (2a_1+N)\sum_{i=1}^4 a_i + 8a_1+4N\geqslant 8a_1+4N\geqslant 4\sum_{i=1}^4 a_i,
$$
because $a_1\geqslant a_3\geqslant a_4$. Thus $N=a_2$, $a_3+a_4>1, a_2+a_4>1$ and $a_2+a_3>1+a_4$.
Let $a_5=0$.
Now the result follows from the first inequality,
since $a_2$ is the largest number not exceeding $1$ among those values in the hypothesis.
\end{proof}

\end{document}